\documentclass[reqno]{amsart}

\usepackage{amsmath}
\usepackage{amsfonts}
\usepackage{amssymb,enumerate}
\usepackage{amsthm}
\usepackage[all]{xy}
\usepackage{rotating}
\usepackage{hyperref}
\usepackage{color,bbm}

\theoremstyle{plain}
\newtheorem{lem}{Lemma}[section]
\newtheorem{cor}[lem]{Corollary}
\newtheorem{prop}[lem]{Proposition}
\newtheorem{thm}[lem]{Theorem}

\theoremstyle{definition}

\newtheorem{ex}[lem]{Example}
\newtheorem{question}[lem]{Question}

\newtheorem{construction}[lem]{Construction}

\newtheorem{fact}[lem]{Fact}

\newtheorem{step}[lem]{Step}

\newcommand{\HH}{\operatorname{H}}

\newcommand{\shift}{\mathsf{\Sigma}}

\newcommand{\cone}{\operatorname{Cone}}

\newcommand{\ol}{\overline}
\newcommand{\wti}{\widetilde}

\newcommand{\bbz}{\mathbb{Z}}
\newcommand{\bbn}{\mathbb{N}}

\newcommand{\xra}{\xrightarrow}

\renewcommand{\geq}{\geqslant}
\renewcommand{\leq}{\leqslant}

\usepackage{tikz}
\usepackage[pagewise]{lineno}

\numberwithin{equation}{lem}

\newcommand{\edgeIdeal}[1][G]{\mathcal{E}(#1)}
\newcommand{\newIdeal}[1]{\langle #1 \rangle}

\newcommand{\bbk}{\mathbbm{k}}

\begin{document}

\bibliographystyle{amsplain}

\author{Bethany Kubik}

\author{Denise Rangel Tracy}
\address{Denise Rangel Tracy, Francis Marion University, Department of Mathematics, 4822 E. Palmetto St., Florence, S.C. 29506, PO Box 100547, Florence, S.C. 29502}
\email{rangel.tracy@fmarion.edu}
\urladdr{https://www.fmarion.edu/directory/tracy-dr-denise-a-r/}

\author{Keri Ann Sather-Wagstaff}
\address{Keri Ann Sather-Wagstaff, School of Mathematical and Statistical Sciences,
Clemson University,
O-110 Martin Hall, Box 340975, Clemson, S.C. 29634
USA}
\email{ssather@clemson.edu}
\urladdr{https://ssather.people.clemson.edu/}

\thanks{
Sather-Wagstaff was supported by the National Science Foundation under Grant No.\ EES-2243134.
Any opinions, findings, and conclusions or recommendations expressed in this material are those of the author and do not necessarily reflect the views of the National Science Foundation.}

\title[Resolutions of Edge Ideals of Weighted Complete Bipartite Graphs]{Minimal Free Resolutions of Edge Ideals of Edge-Weighted Complete Bipartite Graphs}

\date{\today}


\keywords{Complete bipartite graphs; minimal free resolutions; weighted edge ideals}
\subjclass[2010]{Primary: 13D02; Secondary: 05C22, 05E40, 13F55}

\begin{abstract}
We explicitly describe cellular minimal free resolutions of certain classes of edge ideals of weighted complete bipartite graphs based on a construction of Visscher. Specifically, we show that Visscher's construction minimally resolves all edge ideals of undirected vertex-weighted complete bipartite graphs, and we characterize the edge-weighted complete bipartite graphs whose edge ideals are minimally resolved by Visscher's construction.
\end{abstract}

\maketitle

\tableofcontents

\section{Introduction} \label{sec230728a}

Minimal free resolutions are important constructions in commutative algebra, encoding valuable information
for algebraic considerations~\cite{avramov:ifr,MR0460383} and for applications to coding theory~\cite{math10122079}, geometry~\cite{MR2103875}, and other fields.
These resolutions are hard to compute in general, even for special classes of ideals.
For instance, this was only recently accomplished by 
Eagon, Miller, and Ordog~\cite{eagon2020minimal} for
arbitrary monomial ideals where one has numerous combinatorial tools at one's disposal.
Note that Hochster~\cite{hochster:cmrcsc} had already described the free modules in this case;
the hard work was to explicitly describe the differentials.

We are interested in resolutions with additional combinatorial structure for ideals that arise from complete bipartite graphs. 
For the additional structure, we look to Bayer and Sturmfels'~\cite{MR1647559} notion of cellular resolutions of monomial ideals, i.e., ideals in the polynomial ring
$S=\bbk[X_1,\ldots,X_d]$ that are generated by monomials in the variables $X_i$, where $\bbk$ is a field.
To construct 
these resolutions, one considers a finite regular cell complex $C$ (i.e., a regular CW-complex with a finite number of cells), labeling the vertices with the monomial generators of the ideal $I$,
and one uses $C$ to describe a bounded chain complex 
\begin{equation}F_C=\qquad \cdots
\xra{\partial^{F_C}_{4}} \underbrace{S^{C_2}}_{=F_{C,3}}
\xra{\partial^{F_C}_{3}} \underbrace{S^{C_1}}_{=F_{C,2}}
\xra{\partial^{F_C}_{2}} \underbrace{S^{C_0}}_{=F_{C,1}}
\xra{\partial^{F_C}_{1}} \underbrace{S}_{=F_{C,0}}\to 0
\label{eq230909a}
\end{equation}
with $H_0(F_C)\cong S/I$. Here
each $F_{C,i}=S^{C_{i-1}}$ is a finite-rank free $S$ module with $C_i$ the set of $i$-dimensional faces of $C$, e.g., with $C_{-1}=\{\emptyset\}$ so $F_{C,0}\cong S$.
The differential $\partial^{F_C}_i$ is informed by sign conventions from algebraic topology and the generators of $I$.
See Section~\ref{sec230728b} for more background discussion, including an explicit description of $\partial^{F_C}$ in Construction~\ref{const231022a}. 
Here are two examples of this that we revisit throughout the paper.

\begin{ex}\label{ex230728a}
Consider the square-free monomial ideal 
$$I=\newIdeal{X_1Y_1,X_1Y_2,X_2Y_1,X_2Y_2}\subseteq S=\bbk[X_1,X_2,Y_1,Y_2]$$ 
and the cell complex 
\begin{equation}
\begin{tikzpicture}[baseline=(current  bounding  box.east)]
     \node (L) at (-1, 1) {$C = $};
    \fill[gray!30, rounded corners=10pt] (0,0) rectangle (3,2);    
    \node (A) at (0, 2) {$X_1Y_1$};
    \node (B) at (3, 2) {$X_1Y_2$};
    \node (C) at (0, 0) {$X_2Y_1$};
    \node (D) at (3, 0) {$X_2Y_2$};

    \draw (A) -- (B);
    \draw (A) -- (C);
    \draw (C) -- (D);
    \draw (B) -- (D);

\end{tikzpicture} \hfill\label{eq230806a}
\end{equation}
which is a shaded square with vertices labeled with the generators of $I$. 
Note that $C$ has four cells in dimension 0 (vertices), four cells in dimension 1 (edges), and one cell in dimension 2 (shaded square), in addition to the degenerate empty cell in dimension -1. 
With~\eqref{eq230909a}, this explains the ranks of the free modules in $F_C$:
$$F_C=\quad
0\to
S^1\xra{\!\!\left(\begin{smallmatrix}-X_2\\X_1\\-Y_2\\Y_1\end{smallmatrix}\right)\!\!}
S^4\xra{\!\!\left(\begin{smallmatrix}Y_2&0&-X_2&0\\ -Y_1&0&0&-X_2\\ 0&Y_2&X_1&0\\ 0&-Y_1&0&X_1 \end{smallmatrix}\right)\!\!}
S^4\xra{\left(\begin{smallmatrix}X_1Y_1&X_1Y_2&X_2Y_1&X_2Y_2\end{smallmatrix}\right)}
S\to 0.
$$
For some explanation of the entries in the differential, we add more labels to $C$. 
\begin{center}
\begin{tikzpicture}[baseline=(current  bounding  box.east)]
     \node (L) at (-2, 1) {$C =$};
    \fill[gray!30, rounded corners=10pt] (0,0) rectangle (3,2);    
    \node (A) at (0, 2) {$X_1Y_1$};
    \node (B) at (3, 2) {$X_1Y_2$};
    \node (C) at (0, 0) {$X_2Y_1$};
    \node (D) at (3, 0) {$X_2Y_2$};
    \node (E) at (1.5,1) {\tiny{$X_1X_2Y_1Y_2$}};
    
    \draw (A) -- node[midway, above] {\tiny{$X_1Y_1Y_2$}} (B);
    \draw (A) -- node[midway, left] {\tiny{$X_1X_2Y_1$}} (C);
    \draw (C) -- node[midway, below] {\tiny{$X_2Y_1Y_2$}} (D);
    \draw (B) -- node[midway, right] {\tiny{$X_1X_2Y_2$}} (D);

\end{tikzpicture} 
\end{center}

In addition to the labels on the vertices, each edge is labeled with the LCM of the incident vertices, as is the 2-cell.
Applying the differential to the basis vector one associated to the top horizontal edge, yields a linear combination of the basis vectors associated to the incident vertices.
The corresponding coefficients are the quotients of the associated labels ($X_1Y_1Y_2/X_1Y_1=Y_2$ and $X_1Y_1Y_2/X_1Y_2=Y_1$) with appropriate signs.
One sees these coefficients in the first column of the $4\times 4$ differential $\partial^{F_C}_2$.
See Constructions~\ref{con230806a} and~\ref{const231022a} plus Example~\ref{ex230806a} below for more justification of the differential.
One checks readily that $F_C$ is a minimal cellular resolution of $S/I$; this also follows from a result of Visscher, see Fact~\ref{fact230806a}.

In this paper, we are interested in generalizations of this situation.
Note that the generators of $I$ correspond to the edges of the complete bipartite graph $K_{2,2}$, 
and they are arranged on $C$ so that generators $X_iY_j$ and $X_pY_q$ are adjacent in $C$ if and only if $i=p$ or $j=q$, i.e., if and only if the edges 
$X_iY_j$ and $X_pY_q$ are incident in $K_{2,2}$. 
Our work here involves ideals of the following generalized form
$$J=\newIdeal{X_1^{a_{1,1}}Y_1^{a_{1,1}},X_1^{a_{1,2}}Y_2^{a_{1,2}},X_2^{a_{2,1}}Y_1^{a_{2,1}},X_2^{a_{2,2}}Y_2^{a_{2,2}}}$$ 
with the cell complex $C$ labeled accordingly:
\begin{equation}
\begin{tikzpicture}[baseline=(current  bounding  box.east)]
     \node (L) at (-2, 1) {$C =$};
    \fill[gray!30, rounded corners=10pt] (0,0) rectangle (3,2);    
    \node (A) at (0, 2) {$X_1^{a_{1,1}}Y_1^{a_{1,1}}$};
    \node (B) at (3, 2) {$X_1^{a_{1,2}}Y_2^{a_{1,2}}$};
    \node (C) at (0, 0) {$X_2^{a_{2,1}}Y_1^{a_{2,1}}$};
    \node (D) at (3, 0) {$X_2^{a_{2,2}}Y_2^{a_{2,2}}$};

    \draw (A) -- (B);
    \draw (A) -- (C);
    \draw (C) -- (D);
    \draw (B) -- (D);
\end{tikzpicture} \hfill \label{eq230806b}
\end{equation}
In this case, our main result, Theorem~\ref{thm230730a}, shows the following: 
\begin{enumerate}[\rm(*)]
\item 
\emph{the corresponding chain complex $F_C$ is a resolution of $S/J$ if and only if, up to symmetry, $J$ has the form
$$J=\newIdeal{X_1^{\alpha}Y_1^{\alpha},X_1^{\alpha}Y_2^{\alpha},X_2^{\beta}Y_1^{\beta},X_2^{\gamma}Y_2^{\gamma}}$$ 
with $\alpha\leq\beta\leq \gamma$.}
\end{enumerate}
Then the full labeling of $C$ is
\begin{equation}
\begin{tikzpicture}[baseline=(current  bounding  box.east)]
     \node (L) at (-2, 1) {$C =$};
    \fill[gray!30, rounded corners=10pt] (0,0) rectangle (3,2);    
     \node (A) at (0, 2) {$X_1^{\alpha}Y_1^{\alpha}$};
    \node (B) at (3, 2) {$X_1^{\alpha}Y_2^{\alpha}$};
    \node (C) at (0, 0) {$X_2^{\beta}Y_1^{\beta}$};
    \node (D) at (3, 0) {$X_2^{\gamma}Y_2^{\gamma}$};
   
    \node (E) at (1.5,1) {\tiny{$X_1^{\alpha}X_2^{\gamma}Y_1^{\beta}Y_2^\gamma$}};
    
    \draw (A) -- node[midway, above] {\tiny{$X_1^{\alpha}Y_1^{\alpha}Y_2^{\alpha}$}} (B);
    \draw (A) -- node[midway, left] {\tiny{$X_1^{\alpha}X_2^{\beta}Y_1^{\beta}$}} (C);
   \draw (C) -- node[midway, below] {\tiny{$X_2^{\gamma}Y_1^{\beta}Y_2^{\gamma}$}} (D);
    \draw (B) -- node[midway, right] {\tiny{$X_1^{\alpha}X_2^{\gamma}Y_2^{\gamma}$}} (D);

\end{tikzpicture} \hfill\label{eq230909b}
\end{equation}
and the resolution is
\begin{equation}
\label{eq231022d}
F_C=\qquad
0\to
S^1\xra{\partial^{F_C}_3}
S^4\xra{\partial^{F_C}_2}
S^4\xra{\left(\begin{smallmatrix}X_1^{\alpha}Y_1^{\alpha}&X_1^{\alpha}Y_2^{\alpha}&X_2^{\beta}Y_1^{\beta}&X_2^{\gamma}Y_2^{\gamma}\end{smallmatrix}\right)}
S\to 0.
\end{equation}
with the following differentials in degrees 3 and 2.
\begin{align*}
\partial^{F_C}_3
&=\left(\begin{smallmatrix}-X_2^\gamma Y_1^{\beta-\gamma}Y_2^{\gamma-\alpha}\\X_1^{\alpha}\\-X_2^{\gamma-\beta}Y_2^\gamma \\Y_1^\beta\end{smallmatrix}\right)
&\partial^{F_C}_2
&=\left(\begin{smallmatrix}Y_2^\alpha&0&-X_2^\beta Y_1^{\beta-\alpha}&0\\ -Y_1^\alpha&0&0&-X_2^\gamma Y_2^{\gamma-\alpha}\\ 0&X_2^{\gamma-\beta}Y_2^\gamma&X_1^\alpha&0\\ 
0&-Y_1^\beta&0&X_1^\alpha \end{smallmatrix}\right)
\end{align*} 
See examples below for more about these ideals.
\end{ex}

\begin{ex}\label{ex230728b}
Consider the next monomial ideals in $S=\bbk[X_1,X_2,Y_1,Y_2,Y_3]$
\begin{align*}
I&=\newIdeal{X_1Y_1,X_1Y_2,X_1Y_3,X_2Y_1,X_2Y_2,X_2Y_3}\\
J&=\newIdeal{X_1^\alpha Y_1^\alpha ,X_1^\alpha Y_2^\alpha ,X_1^\alpha Y_3^\alpha ,X_2^\beta Y_1^\beta,X_2^\gamma Y_2^\gamma,X_2^\delta Y_3^\delta}
\end{align*}
with $\alpha\leq\beta\leq\gamma\leq\delta$.
Starting with $I$, we work with the following  cell complex 
\begin{equation}
\begin{tikzpicture}
     \node (L) at (-1.5, 1) {$C = $};
     
       \fill[gray!30, rounded corners=5pt] (0,0) rectangle (3,2);
     \fill[gray!50, rounded corners=5pt] (0,2) -- (1.5,1) -- (3,2) -- cycle;      
     \fill[gray!50, rounded corners=5pt] (0,0) -- (1.5,-1) -- (3,0) -- cycle;

    \node (A) at (0, 2) {$X_1Y_1$};
    \node (B) at (3, 2) {$X_1Y_2$};
    \node (C) at (0, 0) {$X_2Y_1$};
    \node (D) at (3, 0) {$X_2Y_2$};
    \node (E) at (1.5, 1) {$X_1Y_3$};
    \node (F) at (1.5, -1) {$X_2Y_3$};

    \draw (A) -- (B);
    \draw (A) -- (C);
    \draw (B) -- (D);
    \draw[dotted] (C) -- (D);
    
     \draw (A) -- (E);
    \draw (B) -- (E);
    \draw (C) -- (F);    
     \draw (E) -- (F);
    \draw (D) -- (F);
\end{tikzpicture}\hfill\label{eq230806c}
\end{equation}
which is a solid cylindrical prism with triangular base with vertices labeled with the generators of $I$.
The complex $C$ has six cells in dimension 0 (vertices), nine cells in dimension 1 (edges), five cells in dimension 2 (triangles and squares), and one cell in dimension 3 (solid prism), in addition to the degenerate empty cell in dimension -1. 
This explains the ranks of the free modules in $F_C$:
$$F_C=\quad
0\to
S^1\xra{\partial^{F_C}_4}
S^5\xra{\partial^{F_C}_3}
S^9\xra{\partial^{F_C}_2}
S^6\xra{\partial^{F_C}_1}
S\to 0.
$$
The differentials are given by the following matrices:
\begin{align*}
\partial^{F_C}_4
&=\begin{pmatrix}
-X_2\\ X_1\\ Y_3\\ -Y_2\\ Y_1\end{pmatrix}
\\
\partial^{F_C}_3
&=\begin{pmatrix}
-Y_3&0&-X_2&0&0\\
Y_2&0&0&-X_2&0\\
-Y_1&0&0&0&-X_2\\
0&-Y_3&X_1&0&0\\
0&Y_2&0&X_1&0\\
0&-Y_1&0&0&X_1\\
0&0&-Y_2&-Y_3&0\\
0&0&Y_1&0&-Y_3\\
0&0&0&Y_1&Y_2
\end{pmatrix}
\\
\partial^{F_C}_2
&=\begin{pmatrix}
Y_2&Y_3&0&0&0&0&-X_2&0&0 \\
-Y_1&0&Y_3&0&0&0&0&-X_2&0\\
0&-Y_1&-Y_2&0&0&0&0&0&-X_2\\
0&0&0&Y_2&Y_3&0&X_1&0&0&\\
0&0&0&-Y_1&0&Y_3&0&X_1&0\\
0&0&0&0&-Y_1&-Y_2&0&0&X_1
\end{pmatrix}
\\
\partial^{F_C}_1
&=\begin{pmatrix}X_1Y_1&X_1Y_2&X_1Y_3&X_2Y_1&X_2Y_2&X_2Y_3\end{pmatrix}
\end{align*}
As in Example~\ref{ex230728a}, one checks readily or applies Fact~\ref{fact230806a} to see that $F_C$ is a minimal cellular resolution of $S/I$.

Turning to the ideal 
$$J=\newIdeal{X_1^\alpha Y_1^\alpha ,X_1^\alpha Y_2^\alpha ,X_1^\alpha Y_3^\alpha ,X_2^\beta Y_1^\beta,X_2^\gamma Y_2^\gamma,X_2^\delta Y_3^\delta}$$
with $\alpha\leq\beta\leq\gamma\leq\delta$,
these constructions yield the following.
\begin{equation}
\begin{tikzpicture}
     \node (L) at (-1.5, 1) {$C = $};
     
       \fill[gray!30, rounded corners=5pt] (0,0) rectangle (3,2);
     \fill[gray!50, rounded corners=5pt] (0,2) -- (1.5,1) -- (3,2) -- cycle;      
     \fill[gray!50, rounded corners=5pt] (0,0) -- (1.5,-1) -- (3,0) -- cycle;

    \node (A) at (0, 2) {$X_1^\alpha Y_1^\alpha$};
    \node (B) at (3, 2) {$X_1^\alpha Y_2^\alpha$};
    \node (C) at (0, 0) {$X_2^\beta Y_1^\beta$};
    \node (D) at (3, 0) {$X_2^\gamma Y_2^\gamma$};
    \node (E) at (1.5, 1) {$X_1^\alpha Y_3^\alpha$};
    \node (F) at (1.5, -1) {$X_2^\delta Y_3^\delta $};

    \draw (A) -- (B);
    \draw (A) -- (C);
    \draw (B) -- (D);
    \draw[dotted] (C) -- (D);
    
     \draw (A) -- (E);
    \draw (B) -- (E);
    \draw (C) -- (F);    
     \draw (E) -- (F);
    \draw (D) -- (F);
\end{tikzpicture}\hfill\label{eq230806d}
\end{equation}
$$F_C=\quad
0\to
S^1\xra{\partial^{F_C}_4}
S^5\xra{\partial^{F_C}_3}
S^9\xra{\partial^{F_C}_2}
S^6\xra{\partial^{F_C}_1}
S\to 0.
$$
\begin{align*}
\partial^{F_C}_4
&=\begin{pmatrix}
-X_2^\delta Y_1^{\beta-\alpha}Y_2^{\gamma-\alpha}Y_3^{\delta-\alpha}\\ X_1^\alpha\\ X_2^{\delta-\gamma}Y_3^\delta\\ -Y_2^\gamma\\ Y_1^\beta\end{pmatrix}
\\
\partial^{F_C}_3
&=\begin{pmatrix}
-Y_3^\alpha&0&-X_2^\gamma Y_1^{\beta-\alpha}Y_2^{\gamma-\alpha}&0&0\\
Y_2^\alpha&0&0&-X_2^\delta Y_1^{\beta-\alpha}Y_3^{\delta-\alpha}&0\\
-Y_1^\alpha&0&0&0&\!\!\!-X_2^\delta Y_2^{\gamma-\alpha}Y_3^{\delta-\alpha}\\
0&-X_2^{\delta-\gamma}Y_3^\delta&X_1^\alpha&0&0\\
0&Y_2^\gamma&0&X_1^\alpha&0\\
0&-Y_1^\beta&0&0&X_1^\alpha\\
0&0&-X_2^{\gamma-\beta}Y_2^\gamma&-X_2^{\delta-\beta}Y_3^\delta&0\\
0&0&Y_1^\beta&0&-X_2^{\delta-\gamma}Y_3^\delta\\
0&0&0&Y_1^\beta&Y_2^\gamma
\end{pmatrix}
\\
\partial^{F_C}_2
&=\left(\begin{smallmatrix}
Y_2^\alpha&Y_3^\alpha&0&0&0&0&-X_2^\beta Y_1^{\beta-\alpha}&0&0 \\
-Y_1^\alpha&0&Y_3^\alpha&0&0&0&0&-X_2^\gamma Y_2^{\gamma-\alpha}&0\\
0&-Y_1^\alpha&-Y_2^\alpha&0&0&0&0&0&-X_2^\delta Y_3^{\delta-\alpha}\\
0&0&0&X_2^{\gamma-\beta}Y_2^\gamma&X_2^{\delta-\beta}Y_3^\delta&0&X_1^\alpha&0&0&\\
0&0&0&-Y_1^\beta&0&X_2^{\delta-\gamma}Y_3^\delta&0&X_1^\alpha&0\\
0&0&0&0&-Y_1^\beta&-Y_2^\gamma&0&0&X_1^\alpha
\end{smallmatrix}\right)
\\
\partial^{F_C}_1
&=\begin{pmatrix}X_1^\alpha Y_1^\alpha &X_1^\alpha Y_2^\alpha &X_1^\alpha Y_3^\alpha &X_2^\beta Y_1^\beta&X_2^\gamma Y_2^\gamma&X_2^\delta Y_3^\delta\end{pmatrix}
\end{align*}
Our main result Theorem~\ref{thm230730a} shows that $F_C$ is a minimal cellular resolution of $S/J$,
and moreover characterizes the ideals of the form
$$\newIdeal{X_1^{a_{1,1}}Y_1^{a_{1,1}},X_1^{a_{1,2}}Y_2^{a_{1,2}},X_1^{a_{1,3}}Y_3^{a_{1,3}},X_2^{a_{2,1}}Y_1^{a_{2,1}},X_2^{a_{2,2}}Y_2^{a_{2,2}},X_2^{a_{2,3}}Y_3^{a_{2,3}}}$$ 
for which $F_C$ is a minimal cellular resolution.

See examples below for more about these ideals.
\end{ex}

Section~\ref{sec230728b} contains background information about Bayer and Sturmfels' construction, including their criterion for $F_X$ to be acyclic, i.e., 
to be a resolution of $S/I$.

The ideals we consider in this paper start with Villarreal's~\cite{villarreal:cmg} \emph{edge ideal} of a finite simple graph $G$ with vertex set $V=\{X_1,\ldots,X_d\}$.
This is none other than the ideal of $S$ generated by the edges of $G$:
$$\edgeIdeal=\newIdeal{X_iX_j\mid\text{$X_iX_j$ is an edge of $G$}}.$$
As Villarreal and others show, $\edgeIdeal$ 
captures valuable information about $G$; and, conversely, the structure of $G$ gives graph-theoretic explanations of useful algebraic information about $\edgeIdeal$. 
For instance, one can use the ``vertex covers'' of $G$ to find the minimal prime decomposition of $\edgeIdeal$ graph-theoretically~\cite{villarreal:cmg} 
and hence the Krull dimension of $S/\edgeIdeal$.
One can understand the Cohen-Macaulay property and the minimal free resolution for $S/\edgeIdeal$ in some cases 
(see, e.g., \cite{ha:rsfmifi,ha:mieihgbn,MR2231097,vantuyl:cmcg,villarreal:cmg}), 
though these depend not only on $G$ but also on the characteristic of $\bbk$, the ground field~\cite{MR2209703}. 
In particular, one cannot expect to find a regular cell complex supporting the minimal free resolution of $S/\edgeIdeal$ that only depends on $G$ in general,
so we consider special classes of graphs to see when such cell complexes exist. 

Visscher~\cite{MR2262383}
explicitly constructed a cellular minimal free resolution of the square-free edge ideal of an arbitrary complete bipartite graph. 
This cellular resolution is the focus of this paper, so we describe it in some detail here. 

\begin{construction}\label{con230806a}
Fix integers $m,n\geq 1$, and consider the complete bipartite graph $K_{m,n}$ with partite sets $X_1,\ldots,X_m,Y_1,\ldots,Y_n$.
The edge ideal of $K_{m,n}$ is 
$$\edgeIdeal[K_{m,n}]=\newIdeal{X_iY_j\mid i=1,\ldots,m; j=1,\ldots,n}\subseteq S=\bbk[X_1,\ldots,X_m,Y_1,\ldots,Y_n].$$
The nonempty faces of the cell complex $V_{m,n}$ are of the form $(A,B)$ where $\emptyset\neq A\subseteq [m]=\{1,\ldots,m\}$ and $\emptyset\neq B\subseteq [n]$.
The dimension of the face $(A,B)$ is $|A|+|B|-1$.
For instance, the vertices of $V_{m,n}$ are the faces $(\{i\},\{j\})$, which we write as $(i,j)$, for $i\in[m]$ and $j\in[n]$. 
The edges of $V_{m,n}$ are the faces of the form $(\{i,j\},\{p\})=(ij,p)$ and $(\{i\},\{p,q\})=(i,pq)$ where $i\neq j$ and $p\neq q$. 
And so on. 
One labels the vertex $(i,j)\in V_{m,n,0}$ with the generator $X_iY_j\in\edgeIdeal[K_{m,n}]$. 

For the sake of clarity, we explicitly describe the chain complex $F_{m,n}$ including its differential.
The basis in degree 0 is $1\in S=F_{m,n,0}$.
For $d\geq 1$, the basis vectors of $F_{m,n,d}$ correspond to the $(d-1)$-dimensional faces of $V_{m,n}$, i.e., ordered pairs $(A,B)\in V_{m,n}$ with
$\emptyset\neq A\subseteq [m]$ and $\emptyset\neq B\subseteq [n]$ such that $|A|+|B|=d+1$. We denote the corresponding basis vector as $[A,B]\in F_{m,n}$.
Write $A=\{a_1<\cdots<a_s\}$ and $B=\{b_1<\cdots<b_t\}$ with $s,t\geq 2$. Applying the differential to different basis vectors gives
\begin{align*}
\partial^{F_{m,n}}([a,b])&=f_{a,b}
\\
\partial^{F_{m,n}}([A,b])&=\sum_{i=1}^s(-1)^{i-1}X_{a_i}[A-a_i,b]
\\
\partial^{F_{m,n}}([a,B])&=\sum_{j=1}^t(-1)^j Y_{b_j}[a,B-b_j]
\\
\partial^{F_{m,n}}([A,B])&=\sum_{i=1}^s(-1)^{i-1}X_{a_i}[A-a_i,b]+\sum_{j=1}^t(-1)^{s+j-1} Y_{b_j}[A,B-b_j]
\end{align*}


\end{construction}

\begin{ex}\label{ex230806a}
The labeled cell complexes $V_{2,2}$ and $V_{2,3}$ are exactly those in~\eqref{eq230806a} and~\eqref{eq230806c}
which we reproduce here to include both the vertices $(i,j)$ and their labels $X_iY_j$.

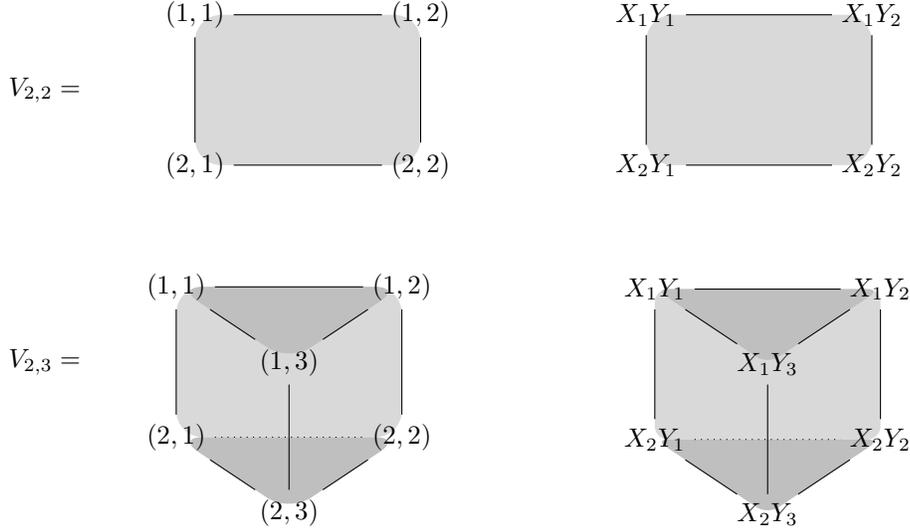
\begin{tikzpicture}[baseline=(current  bounding  box.east)]
     \node (L) at (-2, 1) {$V_{2,2} =$};
    \fill[gray!30, rounded corners=10pt] (0,0) rectangle (3,2);    
    \node (A) at (0, 2) {$(1,1)$};
    \node (B) at (3, 2) {$(1,2)$};
    \node (C) at (0, 0) {$(2,1)$};
    \node (D) at (3, 0) {$(2,2)$};
    
    \draw (A) --  (B);
    \draw (A) --  (C);
    \draw (C) --  (D);
    \draw (B) --  (D);

\end{tikzpicture} \hfill
\begin{tikzpicture}[baseline=(current  bounding  box.east)]
    \fill[gray!30, rounded corners=10pt] (0,0) rectangle (3,2);    
    \node (A) at (0, 2) {$X_1Y_1$};
    \node (B) at (3, 2) {$X_1Y_2$};
    \node (C) at (0, 0) {$X_2Y_1$};
    \node (D) at (3, 0) {$X_2Y_2$};
    
    \draw (A) --  (B);
    \draw (A) --  (C);
    \draw (C) --  (D);
    \draw (B) --  (D);

\end{tikzpicture} \hfill

\vspace{1cm}

\begin{tikzpicture}
     \node (L) at (-1.75, 1) {$V_{2,3} = $};
     
       \fill[gray!30, rounded corners=10pt] (0,0) rectangle (3,2);
     \fill[gray!50, rounded corners=10pt] (0,2) -- (1.5,1) -- (3,2) -- cycle;      
     \fill[gray!50, rounded corners=10pt] (0,0) -- (1.5,-1) -- (3,0) -- cycle;

    \node (A) at (0, 2) {$(1,1)$};
    \node (B) at (3, 2) {$(1,2)$};
    \node (C) at (0, 0) {$(2,1)$};
    \node (D) at (3, 0) {$(2,2)$};
    \node (E) at (1.5, 1) {$(1,3)$};
    \node (F) at (1.5, -1) {$(2,3)$};

    \draw (A) -- (B);
    \draw (A) -- (C);
    \draw (B) -- (D);
    \draw[dotted] (C) -- (D);
    
     \draw (A) -- (E);
    \draw (B) -- (E);
    \draw (C) -- (F);    
     \draw (E) -- (F);
    \draw (D) -- (F);
\end{tikzpicture}\hfill
\begin{tikzpicture}
     \node (L) at (-1.5, 1) {};
     
       \fill[gray!30, rounded corners=5pt] (0,0) rectangle (3,2);
     \fill[gray!50, rounded corners=5pt] (0,2) -- (1.5,1) -- (3,2) -- cycle;      
     \fill[gray!50, rounded corners=5pt] (0,0) -- (1.5,-1) -- (3,0) -- cycle;

    \node (A) at (0, 2) {$X_1Y_1$};
    \node (B) at (3, 2) {$X_1Y_2$};
    \node (C) at (0, 0) {$X_2Y_1$};
    \node (D) at (3, 0) {$X_2Y_2$};
    \node (E) at (1.5, 1) {$X_1Y_3$};
    \node (F) at (1.5, -1) {$X_2Y_3$};

    \draw (A) -- (B);
    \draw (A) -- (C);
    \draw (B) -- (D);
    \draw[dotted] (C) -- (D);
    
     \draw (A) -- (E);
    \draw (B) -- (E);
    \draw (C) -- (F);    
     \draw (E) -- (F);
    \draw (D) -- (F);
\end{tikzpicture}\hfill

Labeling the remaining faces of $V_{2,2}$, one obtains~\eqref{eq230909b}:

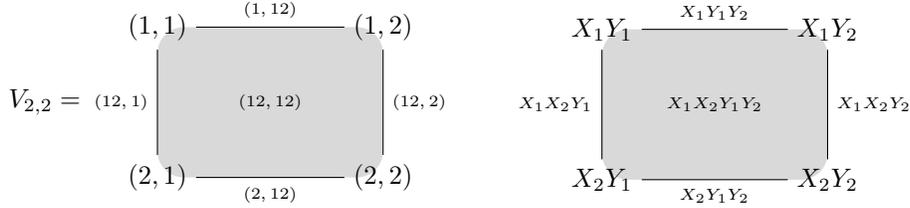
\begin{tikzpicture}
     \node (L) at (-1.5, 1) {$V_{2,2} = $};
    \fill[gray!30, rounded corners=10pt] (0,0) rectangle (3,2);    
   
  \node (A) at (0, 2) {$(1,1)$};
  \node (B) at (3, 2) {$(1,2)$};
 \node (C) at (0, 0) {$(2,1)$};
  \node (D) at (3, 0) {$(2,2)$};
  
 \node (E) at (1.5,1) {\tiny{$(12,12)$}};
    
    \draw (A) -- node[midway, above] {\tiny{$(1,12)$}} (B);
    \draw (A) -- node[midway, left] {\tiny{$(12,1)$}} (C);
   \draw (C) -- node[midway, below] {\tiny{$(2,12)$}} (D);
    \draw (B) -- node[midway, right] {\tiny{$(12,2)$}} (D);

\end{tikzpicture}  \hfill\begin{tikzpicture}
    \fill[gray!30, rounded corners=10pt] (0,0) rectangle (3,2);    
   
     \node (A) at (0, 2) {$X_1Y_1$};
    \node (B) at (3, 2) {$X_1Y_2$};
    \node (C) at (0, 0) {$X_2Y_1$};
    \node (D) at (3, 0) {$X_2Y_2$};
   
    \node (E) at (1.5,1) {\tiny{$X_1X_2Y_1Y_2$}};
    
    \draw (A) -- node[midway, above] {\tiny{$X_1Y_1Y_2$}} (B);
    \draw (A) -- node[midway, left] {\tiny{$X_1X_2Y_1$}} (C);
   \draw (C) -- node[midway, below] {\tiny{$X_2Y_1Y_2$}} (D);
    \draw (B) -- node[midway, right] {\tiny{$X_1X_2Y_2$}} (D);

\end{tikzpicture} 

Basis vectors are of the form $[A,B]$ 
with $A\subseteq[m]$ and $B\subseteq[n]$. 
Example computations of the differential are
\begin{align*}
[12,12]&\mapsto X_1[2,12]-X_2[2,12]+Y_1[12,2]-Y_2[12,1]\\
[1,12]&\mapsto -Y_1[1,2]+Y_2[1,1]
\end{align*}
which one sees in the first two columns displayed in the resolution following~\eqref{eq230806a}.
\end{ex}

\begin{fact}[\protect{Visscher~\cite{MR2262383}}]\label{fact230806a}
For all positive integers $m,n$, the geometric realization of $V_{m,n}$ is a regular cell complex, and 
$F_{V_{m,n}}$ is a cellular minimal resolution of $\edgeIdeal[K_{m,n}]$.
(At this point, we acknowledge explicitly that we do not distinguish notationally between the combinatorial and geometric formulations of Visscher's construction.)
\end{fact}

We are interested in how this situation extends to the following non-square-free situation.
Fix an \emph{edge-weighting} on $G$, i.e., a function $\omega\colon E\to\bbn=\{1,2,3,\ldots\}$ where $E$ is the edge set of $G$; 
this is just a way of equipping each edge of $G$ with a positive integer weight, hence the name. 
A graph $G$ equipped with an edge weighting $\omega$ is an \emph{edge-weighted graph} which we denote $G_{\omega}$. 
(We also consider the natural vertex-weighted situation, though it is not nearly as interesting.)

Paulsen and Sather-Wagstaff~\cite{paulsen:eiwg} 
introduced and studied the edge ideal of edge-weighted graph, which is the non-square-free monomial ideal generated by the weighted edges of $G^\omega$: 
$$\edgeIdeal[G^\omega]=\newIdeal{(X_iX_j)^{\omega(X_iX_j)}\mid\text{$X_iX_j$ is an edge of $G$}}.$$
For example, the ideal $J$ in Example~\ref{ex230728a} is exactly $\edgeIdeal[K_{2,2}^\omega]$ where $a_{i,j}=\omega(X_iY_j)$.
Similarly, the ideal $J$ in Example~\ref{ex230728b} is a special case of $\edgeIdeal[K_{2,3}^\omega]$ where $\omega$ has a specific form.

Typically, non-square-free monomial ideals are harder to understand than square-free ones, in part, due to 
powerful combinatorial tools like the Stanley-Reisner correspondence in the square-free situation.
However, interest has increased recently in these non-square-free edge ideals; see, e.g., Diem, et al.~\cite{diem2023sequentially}, Seyed Fakhari, et al.~\cite{MR4296835},
and Wei~\cite{ShuaiType,wei2023cohenmacaulay}.

\begin{construction}\label{con230909a}
Let $m,n\in\bbn$ be given, and let $\omega$ be an edge-weighting on $K_{m,n}$. 
Let $V_{m,n}^\omega$ denote Vischer's cell complex $V_{m,n}$ with each vertex $(i,j)$ labeled with the monomial generator
$(X_iY_j)^{\omega(X_iY_j)}$.
\end{construction}

As in the unweighted case, we explicitly describe the chain complex $F_{m,n}^\omega$ including its differential.
The basis in degree 0 is $1\in S=F_{m,n,0}^\omega$.
For $d\geq 1$, the basis vectors of $F_{m,n,d}^\omega$ correspond to the $(d-1)$-dimensional faces of $V_{m,n}$, i.e., ordered pairs $(A,B)\in V_{m,n}$ with
$\emptyset\neq A\subseteq [m]$ and $\emptyset\neq B\subseteq [n]$ such that $|A|+|B|=d+1$. We denote the corresponding basis vector as $[A,B]\in F_{m,n}^\omega$.
Write $A=\{a_1<\cdots<a_s\}$ and $B=\{b_1<\cdots<b_t\}$ with $s,t\geq 2$. Applying the differential to different basis vectors gives
\begin{align*}
\partial^{F_{m,n}^\omega}([a,b])&=f_{a,b}
\\
\partial^{F_{m,n}^\omega}([A,b])&=\sum_{i=1}^s(-1)^{i-1}\frac{f_{A,b}}{f_{A-a_i,b}}[A-a_i,b]
\\
\partial^{F_{m,n}^\omega}([a,B])&=\sum_{j=1}^t(-1)^j \frac{f_{a,B}}{f_{a,B-b_j}}[a,B-b_j]
\\
\partial^{F_{m,n}^\omega}([A,B])&=\sum_{i=1}^s(-1)^{i-1}\frac{f_{A,b}}{f_{A-a_i,b}}[A-a_i,b]+\sum_{j=1}^t(-1)^{s+j-1} \frac{f_{a,B}}{f_{a,B-b_j}}[A,B-b_j]
\intertext{where}
f_{a,b}&=(X_aY_b)^{\omega(X_aY_b)}\\
f_{A,b}&=\operatorname{lcm}(f_{a_i,b}\mid i\in[s])\\
f_{a,B}&=\operatorname{lcm}(f_{a,b_j}\mid j\in[t])\\
f_{A,B}&=\operatorname{lcm}(f_{a_i,b_j}\mid i\in[s], j\in[t]).
\end{align*}


When it is convenient, we write $\ol F_{m,n}^\omega$ for the truncation and shift of $F_{m,n}^\omega$ that is a candidate to be a resolution of $\edgeIdeal[K_{m,n}]$.
$$\xymatrix@R=5mm{
F_{m,n}^\omega=
&\ar[r] 
&S^{C_2}\ar[r] 
&S^{C_1}\ar[r] 
&S^{C_0}\ar[r] 
&S\ar[r] 
& 0
\\
\ol F_{m,n}^\omega=
&\ar[r] 
&S^{C_3}\ar[r] 
&S^{C_2}\ar[r] 
&S^{C_1}\ar[r] 
&S^{C_0}\ar[r] 
& 0
}$$

With this background behind us, here is the main result of our paper. It gives an inductive characterization of the weighted complete bipartite graphs such that
Visscher's construction yields a  resolution, noting that it is automatically cellular. It is minimal by Fact~\ref{fact231022b}\eqref{fact231022b2}. 

\begin{thm}\label{thm230730a}
Let $m,n\in\bbn$ be given, and let $\omega$ be an edge-weighting on $K_{m,n}$. 
Set $\alpha=\min\{\omega(X_iY_j)\mid i=1,\ldots,m; j=1,\ldots,n\}$.
Then $F_{m,n}^\omega$ is a (cellular, minimal) resolution of $S/\edgeIdeal[K_{m,n}^\omega]$ if and only if one of the following conditions holds:
\begin{enumerate}[\rm 1.]
\item\label{thm230730a1} $m=1$ or $n=1$, or
\item\label{thm230730a2} $m,n\geq 2$ and there is a vertex $v$ such that 
\begin{enumerate}[\rm(\ref{thm230730a}.\ref{thm230730a2}.a)]
\item $\omega(vw)=\alpha$ for all $w$ adjacent to $v$, and
\item Visscher's construction for $K_{m,n}^\omega-v$ yields a (cellular, minimal) resolution.
\end{enumerate}
\end{enumerate}
\end{thm}

\begin{ex}\label{ex230909a}
This theorem states that
$F_{1,n}^\omega$  is always a resolution of $S/\edgeIdeal[K_{1,n}^\omega]$, and similarly for $F_{m,1}^\omega$.
It is straightforward to show that this is the Taylor resolution in this case, so the Taylor resolution is minimal here. 

The theorem makes it easy to find $K_{2,2}^\omega$'s for which Visscher's construction does not give a resolution; for instance, if all the edge-weights are distinct,
then condition~(\ref{thm230730a}.\ref{thm230730a2}.a) fails. Here are some other cases:
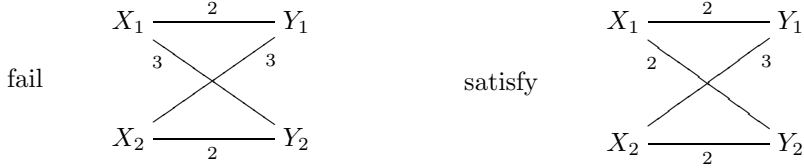
\begin{align*}
\xymatrix@R=3mm@C=7mm{&X_1\ar@{-}[rr]^2\ar@{-}[rrdd]_<<<3&&Y_1\ar@{-}[lldd]^<<<3\\
\text{fail} \\
&X_2\ar@{-}[rr]_2&&Y_2}
&&&&
\xymatrix@R=3mm@C=7mm{&X_1\ar@{-}[rr]^2\ar@{-}[rrdd]_<<<2&&Y_1\ar@{-}[lldd]^<<<3\\
\text{satisfy} \\
&X_2\ar@{-}[rr]_2&&Y_2}
\end{align*}
Moreover, this theorem explains the statement~(*) in Example~\ref{ex230728a}.
It also shows why Visscher's construction gives a resolution for the ideals $J$ in Example~\ref{ex230728b}.
\end{ex} 

The proof of Theorem~\ref{thm230730a} is the content of Section~\ref{sec230728d}.
See Corollary~\ref{cor240505a} for the corresponding computation of Betti numbers.

\section{Background on Cellular Resolutions} \label{sec230728b}

Here we recall Bayer and Sturmfels'~\cite{MR1647559} construction of cellular chain complexes $F_C$ and their criteria for acyclicity (i.e., when $F_C$ is a resolution) and minimality.
Readers may also wish to consult Bruns and Herzog~\cite[Section~6.2]{bruns:cmr}. 
The section ends with our proof that Visscher's construction yields a minimal cellular resolution in the undirected vertex-weighted setting. 

A \emph{finite regular cell complex} is a topological space $C\neq\emptyset$ written as a finite union of pairwise disjoint \emph{open cells} $E_0,\ldots,E_t$ such that
each $E_i$ is homeomorphic to an open euclidean ball such that the closure $\ol{E_i}$ is compatibly homeomorphic to the corresponding closed euclidean ball. 
We assume that $E_0=\emptyset$, the unique cell of dimension $-1$. 
The \emph{vertices} of $C$ are the cells of dimension $0$, which we assume are $E_1,\ldots,E_s$.
The \emph{faces} of $E_i$ are the cells of $C$ contained in the closure $\ol{E_i}$.
The \emph{vertices} of a cell $E_i$ are the vertices of $C$ contained in the closure $\ol{E_i}$, i.e., these are the faces of $E_i$ of dimension 0.
Some examples are contained in Section~\ref{sec230728a}, in particular, Visscher's construction $V_{m,n}$ is a regular cell complex~\cite{MR2262383},
as is any finite topological simplicial complex.

Let $f_1,\ldots,f_s$ be monomials in the polynomial ring $S$, and let $C$ be a finite regular cell complex with cells as above.
The corresponding \emph{labeling} of $C$ is the function $\ell\colon\{E_0,\ldots,E_t\}\to S$ given by 
\begin{equation}
\ell(E_i)=\operatorname{lcm}(f_j\mid\text{$E_j$ is a vertex of $E_i$}).\label{eq231022a}
\end{equation}
The associated \emph{labeled cell complex} is the ordered pair $(C,\ell)$, which we usually write simply as $C$.
Again, see Section~\ref{sec230728a} for examples.
Given a monomial $f\in S$, we consider the following subcomplex of $C$:
$$C_{\leq f}=\bigcup_{\ell(E_i)\mid f}E_i.
$$
In other words, $C_{\leq f}$ is the subcomplex of $C$ induced by the vertices $E_i$ with $f_i\mid f$.

\begin{ex}\label{ex231022a}
Consider the labeled cell complex~\eqref{eq230909b} with $\alpha\leq\beta\leq\gamma$:\\
\begin{center}
\begin{tikzpicture}[baseline=(current  bounding  box.east)]
     \node (L) at (-2, 1) {$C =$};
    \fill[gray!30, rounded corners=10pt] (0,0) rectangle (3,2);    
     \node (A) at (0, 2) {$X_1^{\alpha}Y_1^{\alpha}$};
    \node (B) at (3, 2) {$X_1^{\alpha}Y_2^{\alpha}$};
    \node (C) at (0, 0) {$X_2^{\beta}Y_1^{\beta}$};
    \node (D) at (3, 0) {$X_2^{\gamma}Y_2^{\gamma}$};
   
    \node (E) at (1.5,1) {\tiny{$X_1^{\alpha}X_2^{\gamma}Y_1^{\beta}Y_2^\gamma$}};
    
    \draw (A) -- node[midway, above] {\tiny{$X_1^{\alpha}Y_1^{\alpha}Y_2^{\alpha}$}} (B);
    \draw (A) -- node[midway, left] {\tiny{$X_1^{\alpha}X_2^{\beta}Y_1^{\beta}$}} (C);
   \draw (C) -- node[midway, below] {\tiny{$X_2^{\gamma}Y_1^{\beta}Y_2^{\gamma}$}} (D);
    \draw (B) -- node[midway, right] {\tiny{$X_1^{\alpha}X_2^{\gamma}Y_2^{\gamma}$}} (D);

\end{tikzpicture} \hfill
\end{center}
With the monomial $f=X_1^\beta X_2^\beta Y_1^\beta$, for instance, we have
$$\xymatrix@R=3mm@C=7mm{&&X_1^{\alpha}Y_1^{\alpha}\ar@{-}[rr]^-{X_1^{\alpha}Y_1^{\alpha}Y_2^{\alpha}}\ar@{-}[dd]_-{X_1^{\alpha}X_2^{\beta}Y_1^{\beta}}&&X_1^{\alpha}Y_2^{\alpha}\\
C_{\leq X_1^\beta X_2^\beta Y_1^\beta}= &&&\\
&&X_2^{\beta}Y_1^{\beta}.}
$$
\end{ex}

An \emph{incidence function} on  a finite regular cell complex $C$ is a function $\epsilon$ with domain the set of all ordered pairs $(E_i,E_j)$ such that 
$\dim(E_j)=\dim(E_i)$ and with $\epsilon(E_i,E_j)\in\{0,\pm 1\}$ subject to the following:
\begin{enumerate}[(a)]
\item $\epsilon(E_i,E_j)\neq 0$ if and only if $E_j$ is a face of $E_i$;
\item $\epsilon(E_i,E_0)=1$ for each vertex $E_i$; and
\item $\epsilon(E_i,E_j)\epsilon(E_j,E_k)+\epsilon(E_i,E_{j'})\epsilon(E_{j'},E_k)=0$
whenever $\dim(E_k)=\dim(E_i)-2$, and $E_j$,$E_{j'}$ are the codimension-1 faces of $E_i$ having $E_k$ as a face.
\end{enumerate}

\begin{ex}\label{ex231022b}
For the previous example, we have the incidence function coming from the standard counterclockwise orientation of our shaded rectangle:\\
\begin{center}
\begin{tikzpicture}[baseline=(current  bounding  box.east)]
    \fill[gray!30, rounded corners=10pt] (0,0) rectangle (3,2);    
     \node (A) at (0, 2) {$X_1^{\alpha}Y_1^{\alpha}$};
    \node (B) at (3, 2) {$X_1^{\alpha}Y_2^{\alpha}$};
    \node (C) at (0, 0) {$X_2^{\beta}Y_1^{\beta}$};
    \node (D) at (3, 0) {$X_2^{\gamma}Y_2^{\gamma}$};
   
    \node (E) at (1.5,1) {$\circlearrowleft$};
    
    \draw [<-] (A) -- node[midway, above] {\tiny{$+$ \hspace{1.3cm} $ -$}} (B);
    \draw (A)[->] -- node[midway, right] {\tiny{$+$ }} (C);
   \draw[->] (C) -- node[midway, below] {\tiny{$-$ \hspace{1.3cm}$ +$}} (D);
    \draw [<-] (B) -- node[midway, left] {\tiny{ $ +$}} (D);

    \node () at (1.5, .2) {\tiny{$+$}};
    \node () at (1.5, 1.8) {\tiny{$+$}};
    
    \node () at (-.2, .5) {\tiny{$+$}};
    \node () at (3.2, 1.6) {\tiny{$+$}};
     \node () at (-.2, 1.6) {\tiny{$-$}};
    \node () at (3.2, .4) {\tiny{$-$}};

\end{tikzpicture} \hfill
\end{center}
Here is another one that is useful for later
\begin{equation}
\begin{tikzpicture}[baseline=(current  bounding  box.east)]
    \fill[gray!30, rounded corners=10pt] (0,0) rectangle (3,2);    
     \node (A) at (0, 2) {$X_1^{\alpha}Y_1^{\alpha}$};
    \node (B) at (3, 2) {$X_1^{\alpha}Y_2^{\alpha}$};
    \node (C) at (0, 0) {$X_2^{\beta}Y_1^{\beta}$};
    \node (D) at (3, 0) {$X_2^{\gamma}Y_2^{\gamma}$.};
   
    
    \draw [-] (A) -- node[midway, above] {\tiny{$+$ \hspace{1.3cm} $ -$}} (B);
    \draw (A)[-] -- node[midway, right] {\tiny{$-$ }} (C);
   \draw[-] (C) -- node[midway, below] {\tiny{$+$ \hspace{1.3cm}$ -$}} (D);
    \draw [-] (B) -- node[midway, left] {\tiny{ $ +$}} (D);

    \node () at (1.5, .2) {\tiny{$+$}};
    \node () at (1.5, 1.8) {\tiny{$-$}};
    
    \node () at (-.2, .5) {\tiny{$+$}};
    \node () at (3.2, 1.6) {\tiny{$-$}};
     \node () at (-.2, 1.6) {\tiny{$-$}};
    \node () at (3.2, .4) {\tiny{$+$}};

\end{tikzpicture} \hfill\label{eq231022b}
\end{equation}
\end{ex}

With this background information, we can now construct Bayer and Sturmfels cellular chain complex.

\begin{construction}\label{const231022a}
Let $C$ be a finite regular cell complex as above, with incidence function $\epsilon$ and labeled with monomials $f_1,\ldots,f_s$.
For each $i\in\bbz$, let $C_i$ be the set of $(i-1)$-dimensional cells of $C$.
Consider the sequence of free $S$-modules 
\begin{equation*}F_C=\qquad \cdots
\xra{\partial^{F_C}_{4}} \underbrace{S^{C_2}}_{=F_{C,3}}
\xra{\partial^{F_C}_{3}} \underbrace{S^{C_1}}_{=F_{C,2}}
\xra{\partial^{F_C}_{2}} \underbrace{S^{C_0}}_{=F_{C,1}}
\xra{\partial^{F_C}_{1}} \underbrace{S}_{=F_{C,0}}\to 0
\end{equation*}
where, for each basis vector $E_i\in C_k$, we have
$$\partial^{F_C}_k=\sum_{\dim(E_j)=k-1}\epsilon(E_i,E_j)\frac{\ell(E_i)}{\ell(E_j)}E_j$$
where $\ell(E_i)$ is the monomial label on $E_i$ from~\eqref{eq231022a}.
\end{construction}

\begin{ex}\label{ex231022c}
With $C$ as in Example~\ref{ex231022a} and $\epsilon$ as in~\eqref{eq231022b}, one checks readily that $F_C$
is exactly the complex from~\eqref{eq231022d}. For instance, for the second column of $\partial_2^{F_C}$, let $E=(E_3-E_4)$ be the bottom horizontal edge of $C$, and compute:
$$\partial_2^{F_C}(E)=
\frac{X_2^{\gamma}Y_1^{\beta}Y_2^{\gamma}}{X_2^{\beta}Y_1^{\beta}}E_3
-\frac{X_2^{\gamma}Y_1^{\beta}Y_2^{\gamma}}{X_2^{\gamma}Y_2^{\gamma}}E_4
=
X_2^{\gamma-\beta}Y_2^{\gamma}E_3
-Y_1^{\beta}E_4.
$$
\end{ex}

\begin{fact}\label{fact231022b}
Let $C$ be a finite regular cell complex as above, with incidence function $\epsilon$ and labeled with monomials $f_1,\ldots,f_s$.
\begin{enumerate}[(a)]
\item The sequence $F_C$ is a chain complex with $\HH_0(F_C)=S/\newIdeal{f_1,\ldots,f_s}$.
\item Different incident functions on $C$ give isomorphic complexes.
\item \label{fact231022b1}
The complex $F_C$ may or may not be minimal. More specifically, it is minimal if and only if for every face $E_i$ of $C$ and every codimension-1 face $E_j$ of $E_i$, we have 
$\ell(E_i)\neq\ell(E_j)$, since the coefficient $\pm\ell(E_i)/\ell(E_j)$ is a unit if and only if $\ell(E_i)=\ell(E_j)$.
\item \label{fact231022b2}
In particular, the complex $F^\omega_{m,n}$ is always minimal. Indeed, coefficients of the form
$\pm\ell([A,B])/\ell([A-a,B])$ are divisible by $X_a$, and  the coefficients of the form
$\pm\ell([A,B])/\ell([A,B-b])$ are divisible by $Y_b$.
\end{enumerate}
\end{fact}

The complex $F_C$ may or may not be a resolution; see Fact~\ref{fact230909a}. But the following is an important example of when it is. 

\begin{ex}\label{ex231022e}
Let $\Delta$ be the $(s-1)$-dimensional simplex on $s$ vertices $E_1,\ldots,E_s$.
Specify $\epsilon$.
Then $F_\Delta$ is a resolution of  $S/\newIdeal{f_1,\ldots,f_s}$, called the \emph{Taylor resolution}~\cite{taylor:igmrs}.
\end{ex}

We need one more collection of tools to state Bayer and Sturmfels' criterion for $F_C$ to be a resolution.

Let $C$ be a finite regular cell complex as above, with incidence function $\epsilon$.
If one labels $C$ with the monomials $1,1,\ldots,1\in\bbk$, then the homology of $F_C$ is the 
\emph{reduced cellular homology} of $C$, denoted
$\wti\HH_{i}(C)=\HH_{i}(F_C)$.
This is isomorphic to the reduced singular homology of $C$ with coefficients in $\bbk$. 
We say that $C$ is \emph{acyclic} (over $\bbk$) if $\wti\HH_{i}(C)=0$ for all $i$. 

\begin{fact}[\protect{Bayer and Sturmfels~\cite{MR1647559}}]\label{fact230909a}
$F_C$ is a resolution if and only if for each monomial $f\in S$ the subcomplex $C_{\leq f}$ is either empty or acyclic.
\end{fact}

\subsection*{The Vertex-Weighted Case} \label{sec230728c}

We conclude this section by addressing the vertex-weighted case.
By this, we consider the situation where $G$ is equipped with a \emph{vertex-weighting} $\lambda\colon V\to\bbn$
where $V=\{X_1,\ldots,X_d\}$ is the vertex set of $G$.
Let ${}^\lambda G$ denote a vertex-weighted graph, i.e., a graph $G$ equipped with a vertex-weighting $\lambda$. 
The \emph{edge ideal} of ${}^\lambda G$ is the ideal generated by the vertex-weighted edges of $G$:
$$\edgeIdeal[{}^\lambda G]=\newIdeal{X_i^{\lambda(X_i)}X_j^{\lambda(X_j)}\mid\text{$X_iX_j$ is an edge of $G$}}.
$$
(Since $G$ is undirected, this construction is inherently different from the edge ideal of a weighted oriented graph introduced by H\`a, et al.~\cite{MR3955821}.)
We suggestively sketch $^\lambda K_{2,2}$ as follows:
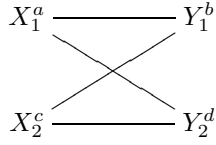
\begin{align*}
\xymatrix@R=3mm@C=7mm{&X_1^a\ar@{-}[rr]\ar@{-}[rrdd]&&Y_1^b\ar@{-}[lldd]\\
\\
&X_2^c\ar@{-}[rr]&&Y_2^d}
\end{align*}
where $a=\lambda(X_1)$ and so on;
the edge ideal in this case is
$$\edgeIdeal[{}^\lambda K_{2,2}]=\newIdeal{X_1^aY_1^b,X_1^aY_2^d,X_2^cY_1^b,X_2^cY_2^d}.$$

Let $m,n\in\bbn$ be given, and let $\lambda$ be a vertex-weighting on $K_{m,n}$. 
Let $^\lambda V_{m,n}$ denote Vischer's cell complex $V_{m,n}$ with each vertex $(i,j)$ labeled with the monomial generator
$X_i^{\lambda(X_i)}Y_j^{\lambda(Y_j)}$.

\begin{prop}\label{prop230909a}
Let $m,n\in\bbn$ be given, and let $\lambda$ be any vertex-weighting on $K_{m,n}$. 
Then Visscher's construction  $^\lambda V_{m,n}$ yields a (cellular, minimal) resolution of $S/\edgeIdeal[{}^\lambda K_{m,n}]$.
\end{prop}

\begin{proof}
One checks readily that for each monomial $f$, the unlabeled subcomplex $(^\lambda V_{m,n})_{\leq f}$ has the form $(V_{m,n})_{\leq g}$ for some  $g$;
this uses the fact that we are weighting the vertices, specifically, because each generator divisible by $X_i$  is maximally divisible by $X_i^{\lambda(X_i)}$,
and similarly for $Y_j$.
Since each $(V_{m,n})_{\leq g}$ is either empty or acyclic by Visscher~\ref{fact230806a}, it follows that each $(^\lambda V_{m,n})_{\leq f}$ is either empty or acyclic,
so $^\lambda V_{m,n}$ yields a resolution by Bayer and Sturmfel's criterion~\ref{fact230909a}. 
The minimality of this resolution is verified like Fact~\ref{fact231022b}\eqref{fact231022b2}.
\end{proof}

\section{Proof of Theorem~\ref{thm230730a}} \label{sec230728d}

Recall that Fact~\ref{fact231022b}\eqref{fact231022b2} shows that $F^\omega_{m,n}$ is always minimal.

\begin{step} \label{step230913a}
First, consider the special case $m=1$. It is straightforward to show that $V_{1,n}$ is an $(n-1)$-dimensional simplex.
It follows that Visscher's construction $V_{1,n}^\omega$ yields the Taylor resolution which is always a resolution. 
This verifies the result when $m=1$, and the case $n=1$ is handled similarly. 
\end{step}

\begin{step} \label{step230913b}
Next, consider the special case $m=n=2$. 
Assume without loss of generality that $\alpha=\omega(X_1Y_1)$, that is, that $\omega(X_1Y_1)\leq\omega(X_iY_j)$ for all $i,j\in[2]$. 
In other words, $K_{2,2}^\omega$ has the following form where $\alpha\leq\beta,\gamma,\delta$.
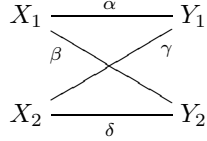
$$\xymatrix@R=3mm@C=7mm{&X_1\ar@{-}[rr]^\alpha\ar@{-}[rrdd]_<<<\beta&&Y_1\ar@{-}[lldd]^<<<\gamma\\
 \\
&X_2\ar@{-}[rr]_\delta&&Y_2}$$
Because of Step~\ref{step230913a} above, we need to show that $V_{2,2}^\omega$ yields a resolution of $S/\edgeIdeal[K_{2,2}^\omega]$
if and only if there is a vertex $v$ such that 
$\omega(vw)=\alpha$ for all $w$ adjacent to $v$.

For the forward implication in this case, we argue by contrapositive.
So suppose there is no such vertex $v$. In particular, this implies that $\alpha<\beta,\gamma$ and $\alpha\leq\delta$.
We inspect $V_{2,2}^\omega$ 
\begin{center}

\begin{tikzpicture}
     \node (L) at (-2, 1) {$V_{2,2}^\omega =$};
    \fill[gray!30, rounded corners=10pt] (0,0) rectangle (3,2);    
     \node (A) at (0, 2) {$X_1^{\alpha}Y_1^{\alpha}$};
    \node (B) at (3, 2) {$X_1^{\beta}Y_2^{\beta}$};
    \node (C) at (0, 0) {$X_2^{\gamma}Y_1^{\gamma}$};
    \node (D) at (3, 0) {$X_2^{\delta}Y_2^{\delta}$};
   
    
    \draw (A) --  (B);
    \draw (A) --  (C);
   \draw (C) --  (D);
    \draw (B) -- (D);

\end{tikzpicture} 
\end{center}
and see that the conditions on the weights $\alpha,\ldots,\delta$ imply that the monomial
$f=X_1^\alpha Y_1^\alpha X_2^\delta Y_2^\delta$ makes $(V_{2,2}^\omega)_{\leq f}$ into two isolated vertices:
$$\xymatrix@R=3mm@C=7mm{&X_1^{\alpha}Y_1^{\alpha}\\
(V_{2,2}^\omega)_{\leq f}= \\
&&&X_2^{\delta}Y_2^{\delta}.} 
$$
This cell complex is disconnected, hence neither empty nor acyclic, so $V_{2,2}^\omega$ does not yield a resolution, as desired.

For the converse in this special case, we argue directly. 
Assume that there is a vertex $v$ such that 
$\omega(vw)=\alpha$ for all $w$ adjacent to $v$.
By symmetry, assume without loss of generality that $v=X_1$,
so $K_{2,2}^\omega$ has the following form where $\alpha\leq\gamma,\delta$.
$$\xymatrix@R=3mm@C=7mm{&X_1\ar@{-}[rr]^\alpha\ar@{-}[rrdd]_<<<\alpha&&Y_1\ar@{-}[lldd]^<<<\gamma\\
 \\
&X_2\ar@{-}[rr]_\delta&&Y_2}$$
Again, we inspect $V_{2,2}^\omega$ 
\begin{center}

\begin{tikzpicture}
     \node (L) at (-2, 1) {$V_{2,2}^\omega =$};
    \fill[gray!30, rounded corners=10pt] (0,0) rectangle (3,2);    
     \node (A) at (0, 2) {$X_1^{\alpha}Y_1^{\alpha}$};
    \node (B) at (3, 2) {$X_1^{\alpha}Y_2^{\alpha}$};
    \node (C) at (0, 0) {$X_2^{\gamma}Y_1^{\gamma}$};
    \node (D) at (3, 0) {$X_2^{\delta}Y_2^{\delta}$.};
   
    
    \draw (A) --  (B);
    \draw (A) --  (C);
   \draw (C) --  (D);
    \draw (B) -- (D);

\end{tikzpicture} 
\end{center}
To apply Bayer and Sturmfel's criterion~\ref{fact230909a}, we assume that $f=X_1^a X_2^b Y_1^c Y_2^d$ is such that
$(V_{2,2}^\omega)_{\leq f}$ is nonempty and prove that it is acyclic. 
Recall that $(V_{2,2}^\omega)_{\leq f}$ is the subcomplex of $V_{2,2}^\omega$ induced by the vertices whose labels divide $f$. 
Because of the small and simple shape of $V_{2,2}^\omega$, it is straightforward to show that the only nonempty, non-acyclic induced subcomplexes of $V_{2,2}^\omega$
are the following:
$$\xymatrix@R=3mm@C=3mm{&X_1^{\alpha}Y_1^{\alpha}\\
C= &&&&&&\text{or}\\
&&&X_2^{\delta}Y_2^{\delta}}
\qquad
 \xymatrix@R=3mm@C=3mm{&&&&X_1^{\alpha}Y_2^{\alpha}\\
&D= \\
&&X_2^{\gamma}Y_1^{\gamma}} 
$$
(For instance, the 1-skeleton is not an \textit{induced} subcomplex since the inclusion of the four edges necessitates the inclusion of the 2-cell they surround.)
We show that these are not of the form $(V_{2,2}^\omega)_{\leq f}$.

Suppose by way of contradiction that $C=(V_{2,2}^\omega)_{\leq f}$. Since $X_1^{\alpha}Y_1^{\alpha}\in C$, we have
$\alpha\leq a$ and $\alpha\leq c$. Similarly, we have $\alpha\leq\delta\leq b$ and $\alpha\leq\delta\leq d$. It follows that $X_1^\alpha Y_2^\alpha\in (V_{2,2}^\omega)_{\leq f}=C$,
contradicting the explicit form of $C$.

A similar argument shows that $D\neq (V_{2,2}^\omega)_{\leq f}$, as desired. This concludes the proof in Step~\ref{step230913b}.
\end{step}

\begin{step} \label{step230913c}
Now, we deal with forward implication in the general case. 
We argue by contrapositive.
By Step~\ref{step230913a}, we assume that $m,n\geq 2$. 

First, assume that condition~(\ref{thm230730a}.\ref{thm230730a2}.a) fails, so there is no vertex $v$ such that 
$\omega(vw)=\alpha$ for all $w$ adjacent to $v$.
That is, for every edge $vw$ with $\omega(vw)=\alpha$, there is a vertex $u$ adjacent to $v$ such that $\omega(vu)>\alpha$, and similarly for $w$.
By symmetry, assume without loss of generality that $\alpha=\omega(X_1Y_1)$, that is, that $\omega(X_1Y_1)\leq\omega(X_iY_j)$ for all $i,j$. 
It follows that we have $i$ and $j$ such that $\alpha>\omega(X_1Y_j)$ and $\alpha>\omega(X_iY_1)$.
Now argue as in the forward implication of Step~\ref{step230913b} to show that 
$f=X_1^\alpha Y_1^\alpha X_i^\delta Y_j^\delta$ with $\delta=\omega(X_iY_j)$
makes $(V_{m,n}^\omega)_{\leq f}$ into two isolated vertices:
$$\xymatrix@R=3mm@C=7mm{&X_1^{\alpha}Y_1^{\alpha}\\
(V_{m,n}^\omega)_{\leq f}= \\
&&&X_i^{\delta}Y_j^{\delta}.} 
$$
This implies that $V_{m,n}^\omega$ does not yield a resolution, as desired.

Next, assume that condition (\ref{thm230730a}.\ref{thm230730a2}.a) is satisfied, but that condition (\ref{thm230730a}.\ref{thm230730a2}.b) fails.
By symmetry, assume without loss of generality that the vertex $v$ guaranteed by condition (\ref{thm230730a}.\ref{thm230730a2}.a) is $v=X_m$. 
Set $\beta=\max\{\omega(X_iY_j)\mid i=1,\ldots,m; j=1,\ldots,n\}$ and $f=(X_1\cdots X_{m-1}Y_1\cdots Y_n)^\beta$. 
It is straightforward to show that the subcomplex $(V_{m,n}^\omega)_{\leq f}$ is precisely $V_{m-1,n}^\mu$ where $\mu$ is the restriction $\mu=\omega|_{K_{m-1,n}}$.
The failure of condition (\ref{thm230730a}.\ref{thm230730a2}.b) 
says that  Visscher's construction for $(V_{m,n}^\omega)_{\leq f}$ does not yield a resolution.
Bayer and Sturmfel's criterion~\ref{fact230909a} says that there is a monomial $g$ such that $[(V_{m,n}^\omega)_{\leq f}]_{\leq g}$ is nonempty and non-acyclic. 
It is straightforward to show that $[(V_{m,n}^\omega)_{\leq f}]_{\leq g}$ has the form $(V_{m,n}^\omega)_{\leq h}$ for some monomial $h$. 
Thus, $(V_{m,n}^\omega)_{\leq h}$ is nonempty and non-acyclic, so another application of~\ref{fact230909a} shows that $V_{m,n}^\omega$ does not yield a resolution.
This concludes the proof of the forward implication in the general case. 
\end{step}

\begin{step} \label{step230913d}
Now, we deal with the converse in the general case. 
We  argue directly, by induction on $m+n$. Steps~\ref{step230913a} and~\ref{step230913b} above cover the base case $m+n\leq 4$.
In the inductive step, we assume that $m,n\geq 2$ and that there is a vertex $v$ such that 
\begin{enumerate}[\rm(a)]
\item $\omega(vw)=\alpha$ for all $w$ adjacent to $v$, and
\item Visscher's construction for $K_{m,n}^\omega-v$ yields a resolution.
\end{enumerate}
By symmetry, we assume without loss of generality that $v=X_i$.
To make the following argument clean, we re-index our $X$-vertices as $X_0,\ldots,X_m$ and assume that $v=X_0$, so that
\begin{enumerate}[\rm(a)]
\item $\omega(X_0Y_j)=\alpha$ for all $j\in[n]$, and
\item Visscher's construction $V_{m,n}^\mu$ for $K_{m+1,n}^\omega-X_0$ yields a resolution.
\end{enumerate}
Here $\mu=\omega|_{K_{m,n}}$.
Note that this implies that $S=\bbk[X_0,\ldots,X_m,Y_1,\ldots,Y_n]$.

We apply a mapping cone argument to show that Visscher's construction for $K_{m+1,n}^\omega$ yields a resolution.
It is worth noting that this mapping cone description is not present in Visscher's work and may be of independent interest even in that square-free context. 

The mapping cone description of Visscher's resolution comes from the following observation.
For each face $(A,B)\in V_{m+1,n}^\omega$ with $\emptyset\neq A\subseteq [m]^+=[0,1,\ldots,m]$ and $\emptyset\neq B\subseteq[n]$, let $[A,B]$ denote the corresponding 
basis vector in the
chain complex $\ol F_{m+1,n}^\omega$, as in Construction~\ref{con230806a}. 
This gives the following three mutually exclusive cases for the basis vectors.
\begin{enumerate}[(1)]
\item \label{item230913a} $0\notin A$. Basis vectors in this case essentially give the chain complex $\ol F_{m,n}^\mu$, which is a resolution of $\edgeIdeal[K_{m,n}^\mu]$, by assumption.
More rigorously, basis vectors in this case give a subcomplex isomorphic to $\ol F_{m,n}^\mu$.
\item \label{item230913b} $A=\{0\}$. 
Basis vectors in this case essentially give a truncated and shifted Taylor resolution $\ol T$ for the ideal $\newIdeal{X_0^\alpha Y_1^\alpha,\ldots,X_0^\alpha Y_n^\alpha}$:
$$\xymatrix@R=5mm@C=15mm{
T=
&\ar[r]^{\partial^T_3} 
&S^{\binom{n}{2}}\ar[r]^{\partial^T_2} 
&S^{\binom{n}1}\ar[r]^{\partial^T_1} 
&S\ar[r] 
& 0
\\
\ol T=
&\ar[r]^{\partial^{\ol T}_3=-\partial^T_4} 
&S^{\binom{n}3}\ar[r]^{\partial^{\ol T}_2=-\partial^T_3} 
&S^{\binom{n}2}\ar[r]^{\partial^{\ol T}_1=-\partial^T_2} 
&S^{\binom{n}1}\ar[r] 
& 0
}$$
More rigorously, basis vectors $[0,B]\in\ol F_{m+1,n}^\omega$ give a subcomplex isomorphic to $\ol T$ where we denote the corresponding basis vector simply as $[B]\in\ol T$.
The resolution $\ol T$ is a truncated, twisted, and shifted Koszul complex on the sequence $Y_1^\alpha,\ldots,Y_n^\alpha$, and the ideal
$\newIdeal{X_0^\alpha Y_1^\alpha,\ldots,X_0^\alpha Y_n^\alpha}$ is none other than the edge ideal $\edgeIdeal[K_{1,n}^\nu]$ where $\nu=\omega|_{K_{1,n}}$, which is the
constant edge-weighting $\alpha$.
\item \label{item230913c} $A\supsetneq\{0\}$. 
Basis vectors in this case are obtained from basis vectors $[A',B]\in\ol F_{m,n}^\mu$ by the association $[A',B]\rightsquigarrow[A'\cup\{0\},B]$ where $A=A-\{0\}$.
Basis vectors in this case mix with those of the previous cases when the differential on $\ol F_{m+1,n}^\omega$ is applied. 
\end{enumerate}
Let's look at this in an example for use in the remainder of the section.

\begin{ex}\label{ex230913a}
Consider the case $m+1=2$ and $n=3$.
In this case, our assumptions show that $\edgeIdeal[K_{3,2}^\omega]$ is essentially the ideal $J$ from Example~\ref{ex230728b}:
$$\edgeIdeal[K_{3,2}^\omega]=\newIdeal{X_0^\alpha Y_1^\alpha ,X_0^\alpha Y_2^\alpha ,X_0^\alpha Y_3^\alpha ,X_1^\beta Y_1^\beta,X_1^\gamma Y_2^\gamma,X_1^\delta Y_3^\delta}$$
with 
\begin{align*}
\beta&=\omega(X_1Y_1)
&\gamma&=\omega(X_1Y_2)
&\delta&=\omega(X_1Y_3).
\end{align*}
Visscher's cell complex and the associated cellular chain complex are
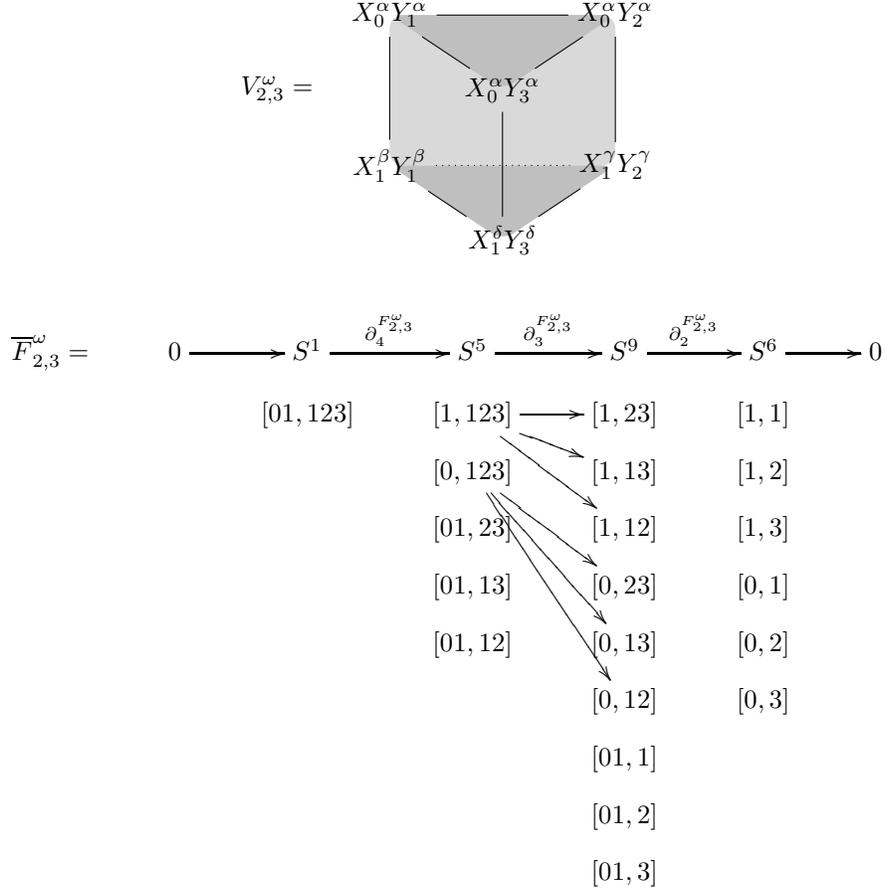
\begin{center}

\begin{tikzpicture}
     \node (L) at (-1.5, 1) {$V_{2,3}^{\omega}= $ };
     
     \fill[gray!30, rounded corners=5pt] (0,0) rectangle (3,2);
     \fill[gray!50, rounded corners=5pt] (0,2) -- (1.5,1) -- (3,2) -- cycle;      
     \fill[gray!50, rounded corners=5pt] (0,0) -- (1.5,-1) -- (3,0) -- cycle;

    \node (A) at (0, 2) {$X_0^\alpha Y_1^\alpha$};
    \node (B) at (3, 2) {$X_0^\alpha Y_2^\alpha$};
    \node (C) at (0, 0) {$X_1^\beta Y_1^\beta$};
    \node (D) at (3, 0) {$X_1^\gamma Y_2^\gamma$};
    \node (E) at (1.5, 1) {$X_0^\alpha Y_3^\alpha$};
    \node (F) at (1.5, -1) {$X_1^\delta Y_3^\delta $};

    \draw (A) -- (B);
    \draw (A) -- (C);
    \draw (B) -- (D);
    \draw[dotted] (C) -- (D);
    
     \draw (A) -- (E);
    \draw (B) -- (E);
    \draw (C) -- (F);    
     \draw (E) -- (F);
    \draw (D) -- (F);
\end{tikzpicture}
\end{center}
$$\xymatrix@R=2mm{
\ol F_{2,3}^\omega=
&0\ar[r]
&S^1\ar[r]^{\partial^{F_{2,3}^\omega}_4}
&S^5\ar[r]^{\partial^{F_{2,3}^\omega}_3}
&S^9\ar[r]^{\partial^{F_{2,3}^\omega}_2}
&S^6\ar[r]
&0\\
&&[01,123]&
[1,123]\ar[r]\ar[rd]\ar[rdd]&
[1,23]&
[1,1]\\
&&&[0,123]\ar[rdd]\ar[rddd]\ar[rdddd]&
[1,13]&
[1,2]\\
&&&[01,23]&
[1,12]&
[1,3]\\
&&&[01,13]&
[0,23]&
[0,1]\\
&&&[01,12]&
[0,13]&
[0,2]\\
&&&&
[0,12]&
[0,3]\\
&&&&
[01,1]&
\\
&&&&
[01,2]&
\\
&&&&
[01,3]&
}
$$
Below each free module, we list the basis vectors with case~\eqref{item230913a} at the top, case~\eqref{item230913b} in the middle, and case~\eqref{item230913c} at the bottom.

Applying the differential to a basis vector of the form $[1,B]$ outputs a linear combination of basis vectors of the same form, as indicated by the arrows emanating from the basis
vector $[1,123]$.
This shows that these basis vectors form a subcomplex of $\ol F_{2,3}^\omega$, which
is part of the point of the statement ``basis vectors in this case give a subcomplex isomorphic to $\ol F_{m,n}^\mu$'' in item~\eqref{item230913a} above.
Here is what this subcomplex looks like in our example.
$$
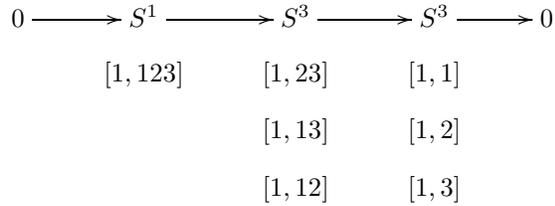
\xymatrix@R=2mm{
&&0\ar[r]
&S^1\ar[r]
&S^3\ar[r]
&S^3\ar[r]
&0\\
&&&[1,123]&
[1,23]&
[1,1]\\
&&&&
[1,13]&
[1,2]\\
&&&&
[1,12]&
[1,3]
}
$$
Note that is is supported on the following subcomplex of $V_{2,3}^\omega$.
\begin{center}
\begin{tikzpicture}
      \fill[gray!30, rounded corners=10pt] (0,1.5) -- (3,1.5) -- (1.5,0) -- cycle;
     
    \node (A) at (-1, .75) {$V_{1,3}^{\mu} = $};
    \node (B) at (0, 1.5) {$X_1^{\beta}Y_1^{\beta}$};
    \node (C) at (3, 1.5) {$X_1^{\gamma}Y_2^{\gamma}$};
    \node (D) at (1.5, 0) {$X_1^{\delta}Y_3^{\delta}$};
   
    
    \draw [->] (B) --  (C);
   \draw (C) --  (D);
    \draw (B) -- (D);
\end{tikzpicture} 
\end{center}
One establishes the remainder of the quoted statement by verifying that the differential on this subcomplex of $F_{m+1,n}^\omega$ is the same as the differential on $\ol F_{m,n}^\mu$.

Similarly, applying the differential to a basis vector of the form $[0,B]$ outputs a linear combination of basis vectors of the same form, as indicated by the arrows emanating from the basis
vector $[0,123]$.
This is part of the point of the statement ``basis vectors $[0,B]\in\ol F_{m+1,n}^\omega$ give a subcomplex isomorphic to $\ol T$'' in item~\eqref{item230913b} above.
Here is what $\ol T$ looks like in this situation.
$$\xymatrix@R=2mm{
\ol T=
&&0\ar[r]
&S^1\ar[r]^{\partial^{\ol T}_2}
&S^3\ar[r]^{\partial^{\ol T}_1}
&S^3\ar[r]
&0\\
&&&[123]&
[23]&
[1]\\
&&&&
[13]&
[2]\\
&&&&
[12]&
[3]
}
$$
Note that is is supported on the following simplex which is the top face of $V_{2,3}^\omega$.
\begin{center}
\begin{tikzpicture}
      \fill[gray!30, rounded corners=10pt] (0,1.5) -- (3,1.5) -- (1.5,0) -- cycle;
     
    \node (B) at (0, 1.5) {$X_0^{\beta}Y_1^{\beta}$};
    \node (C) at (3, 1.5) {$X_0^{\gamma}Y_2^{\gamma}$};
    \node (D) at (1.5, 0) {$X_0^{\delta}Y_3^{\delta}$};
   
    
    \draw [->] (B) --  (C);
   \draw (C) --  (D);
    \draw (B) -- (D);
\end{tikzpicture} 
\end{center}
One establishes 
the remainder of the quoted statement by checking by hand that the differentials on the two complexes are the same.  
For instance, we compute:
\begin{align*}
\partial^{\ol F_{2,3}^\omega}_2([0,123])
&=-\frac{X_0^\alpha Y_1^\alpha Y_2^\alpha Y_3^\alpha}{X_0^\alpha Y_2^\alpha Y_3^\alpha}[0,23]
+\frac{X_0^\alpha Y_1^\alpha Y_2^\alpha Y_3^\alpha}{X_0^\alpha Y_1^\alpha Y_3^\alpha}[0,13]
-\frac{X_0^\alpha Y_1^\alpha Y_2^\alpha Y_3^\alpha}{X_0^\alpha Y_1^\alpha Y_2^\alpha}[0,12]
\\
\partial^{\ol T}_2([123])
&=-\left(\frac{Y_1^\alpha Y_2^\alpha Y_3^\alpha}{Y_2^\alpha Y_3^\alpha}[23]
-\frac{Y_1^\alpha Y_2^\alpha Y_3^\alpha}{Y_1^\alpha Y_3^\alpha}[13]
+\frac{Y_1^\alpha Y_2^\alpha Y_3^\alpha}{Y_1^\alpha Y_2^\alpha}[12]\right)
\end{align*}
which definitely agree; the general computation is similar.
Notice,  it is key here that $\omega(X_0Y_i)=\alpha$ for all $i\in [n]$. 

On the other hand, applying the differential to basis vectors of the form $[01,B]$ outputs linear combinations of basis vectors of all three forms: $[1,B]$, $[0,B]$,  $[01,B']$.
\end{ex}

We return to our general proof of Step~\ref{step230913d}.
In what follows, recall that $\ol F_{m+1,n}^\omega$ lives over the ring $S=\bbk[X_0,\ldots,X_m,Y_1,\ldots,Y_n]$,
and $\ol F_{m,n}^\mu$ lives over the ring $S'=\bbk[X_1,\ldots,X_m,Y_1,\ldots,Y_n]$ with 1 fewer variable.

To describe $\ol F_{m+1,n}^\omega$ as a mapping cone requires a chain map, described on $S$-basis vectors as follows:
$$\xymatrix{
(\ol F_{m,n}^\mu\otimes_{S'}S)(-X_0^\alpha)
\ar[d]_-\phi
&[A,B]\ar@{|->}[d]\\
(\ol F_{m,n}^\mu\otimes_{S'}S)\oplus \ol T& X_0^\alpha\protect{\begin{pmatrix}[A,B]\\ f([A,B])\end{pmatrix}}
}$$
where $(\ol F_{m,n}^\mu\otimes_{S'}S)(-X_0^\alpha)$ is the multi-graded twist of $\ol F_{m,n}^\mu\otimes_{S'}S$ and
\begin{align*}
f([A,B])&=\begin{cases}
0&\text{if $|A|\geq 2$}\\
-\frac{\operatorname{mdeg}([a,B])}{\operatorname{mdeg}([B])}[B]&\text{if $A=\{a\}$}
\end{cases}
\end{align*}
where
\begin{align*}
\operatorname{mdeg}([a,B])
&=X_a^{M_a}\prod_{b\in B}Y_b^{\omega(X_aY_b)}
\\
\operatorname{mdeg}([B])
&=X_0^\alpha\prod_{b\in B}Y_b^\alpha
\end{align*}
where
\begin{align*}
M_a
&=\max(\omega(X_aY_j)\mid j\in B)\geq \alpha.
\end{align*}

Notice that these mdeg's are nothing more than the multidegrees of the multigraded basis vectors in $\ol F_{m,n}^\mu\otimes_{S'}S$ and $\ol T$, respectfully,
that is, the monomial labels on the associated faces of the cell/simplicial complexes supporting these resolutions.

When we take the mapping cone, the basis vectors in case~\eqref{item230913a} and~\eqref{item230913b} will come from the first and second summands in the codomain of $\phi$, 
respectively.
The basis vectors in case~\eqref{item230913c} will come from the domain of $\phi$.

Before continuing, let's look at this for our running example.

\begin{ex}\label{ex230913b}
Continue with the situation of Example~\ref{ex230913a},
and assume that $\beta\leq\gamma\leq\delta$. We write $\phi$ with the Visscher basis for the domain and the Visscher and Taylor bases for the codomain,
suppressing the twists on the free modules. 
$$\xymatrix@R=2mm{
&&[1,123]&
[1,23]&
[1,1]\\
&&&
[1,13]&
[1,2]\\
&&&
[1,12]&
[1,3]
\\
(\ol F_{1,3}^\mu\otimes_{S'}S)(-X_0^\alpha)=
\ar[ddd]_-\phi
&0\ar[r]
&S^1\ar[r]\ar[ddd]_{\phi_2}
&S^3\ar[r]\ar[ddd]_{\phi_1}
&S^3\ar[r]\ar[ddd]_{\phi_0}
&0
\\ \\ \\
(\ol F_{1,3}^\mu\otimes_{S'}S)\oplus \ol T=
&0\ar[r]
&S^{1+1}\ar[r]
&S^{3+3}\ar[r]
&S^{3+3}\ar[r]
&0
\\
&&[1,123]&
[1,23]&
[1,1]\\
&&[123]&
[1,13]&
[1,2]\\
&&&
[1,12]&
[1,3]
\\
&&&
[23]&
[1]\\
&&&
[13]&
[2]\\
&&&
[12]&
[3]
}
$$
Here's what $\phi$ does to a basis vector in the domain.
\begin{align*}
\phi_2([1,12])
&=X_0^\alpha\protect{\begin{pmatrix}[1,12]\\ -\frac{X_1^\gamma Y_1^\beta Y_2^\gamma }{X_0^\alpha Y_1^\alpha Y_2^\alpha}[12]\end{pmatrix}}
=\protect{\begin{pmatrix}X_0^\alpha[1,12]\\ -X_1^\gamma Y_1^{\beta-\alpha} Y_2^{\gamma-\alpha} [12]\end{pmatrix}}\\
\intertext{and we check that the chain map diagram commutes on this basis vector.}
\partial^{(\ol F_{1,3}^\mu\otimes_{S'}S)(-X_0^\alpha)}_2([1,12])
&=-\frac{X_1^\gamma Y_1^\beta Y_2^\gamma}{X_1^\gamma Y_2^\gamma}[1,2]
+\frac{X_1^\gamma Y_1^\beta Y_2^\gamma}{X_1^\beta Y_1^\beta}[1,1] \\
&=-Y_1^\beta [1,2]
+X_1^{\gamma-\beta} Y_2^\gamma[1,1]
\end{align*}
$$\xymatrix@C=19mm{
[1,12]\ar@{|->}[r]^-{\partial_2}\ar@{|->}[d]_-{\phi_2}
& -Y_1^\beta [1,2]
+X_1^{\gamma-\beta} Y_2^\gamma[1,1] \ar@{|->}[d]^-{\phi_1}
\\
\protect{\begin{pmatrix}X_0^\alpha[1,12]\\ -X_1^\gamma Y_1^{\beta-\alpha} Y_2^{\gamma-\alpha} [12]\end{pmatrix}}\ar@{|->}[rd]_-{\left(\begin{smallmatrix}\partial_2&0 \\ 0 & \partial_2\end{smallmatrix}\right)}
&\protect{\begin{pmatrix}X_0^\alpha(-Y_1^\beta[1,2]+X_1^{\gamma-\beta}Y_2^\gamma[1,1])\\ -(-Y_1^\beta\frac{X_1^\gamma Y_2^\gamma}{Y_2^\alpha}[2]+X_1^{\gamma-\beta}Y_2^\gamma\frac{X_1^\beta Y1^\beta}{Y_1^\alpha}[1]) \end{pmatrix}}
\ar@{=}[d]\\
& \protect{\begin{pmatrix} X_0^\alpha(-Y_1^\beta[1,2]+X_1^{\gamma-\beta}Y_2^\gamma[1,1]) \\ -X_1^\gamma Y_1^{\beta-\alpha}Y_2^{\gamma-\alpha}(-Y_1^\alpha[2]+Y_2^\alpha[1])\end{pmatrix}}
}$$
\end{ex}

We return to our general proof of Step~\ref{step230913d}.
As in the preceding example, one checks readily that $\phi$ is a multigraded $S$-linear chain map between multigraded $S$-complexes.

We next show that the mapping cone $\cone(\phi)$ is isomorphic to $F_{m+1,n}^\omega$.
Recall that, given a chain map $\psi\colon W\to U$, we have
$$\cone(\psi)_i=U_i\oplus W_{i-1}\xra[=\left(\begin{smallmatrix}\partial^U_i & \psi_{i-1}\\0&-\partial^W_{i-1}\end{smallmatrix}\right)]{\partial^{\cone(\psi)}_i} U_{i-1}\oplus W_{i-2}=\cone(\psi)_{i-1}.
$$
In our situation, we have $W=(\ol F_{m,n}^\mu\otimes_{S'}S)(-X_0^\alpha)$ and 
$U=(\ol F_{m,n}^\mu\otimes_{S'}S)\oplus \ol T$ so our cone looks like the following.
$$
\xymatrix{
\cone(\phi)_i=
(\ol F_{m,n}^\mu\otimes_{S'}S)_i
\oplus 
\ol T_i
\oplus
(\ol F_{m,n}^\mu\otimes_{S'}S)(-X_0^\alpha)_{i-1}
\ar[dd]^{\left(\begin{smallmatrix}\left(\begin{smallmatrix}\partial^{\ol F_{m,n}^\mu\otimes_{S'}S}_i & 0 \\ 0 & \partial^{\ol T}_i\end{smallmatrix}\right) & \phi_{i-1}\\0&-\partial^{(\ol F_{m,n}^\mu\otimes_{S'}S)(-X_0^\alpha)}_{i-1}\end{smallmatrix}\right)}_{\partial^{\cone(\phi)}_i=} 
\\ \\
\cone(\phi)_{i-1}=(\ol F_{m,n}^\mu\otimes_{S'}S)_{i-1}
\oplus 
\ol T_{i-1}
\oplus 
(\ol F_{m,n}^\mu\otimes_{S'}S)(-X_0^\alpha)_{i-2}}
$$
We have essentially already defined the isomorphism between $\cone(\phi)$ and $F_{m+1,n}^\omega$. 
We summarize here on basis vectors.
$$
\xymatrix@R=1mm{
\cone(\phi)\ar[r]^-\Phi
&\ol F_{m+1,n}^\omega
\\
\protect{\begin{pmatrix} [A,B] \\ 0 \\ 0 \end{pmatrix}}\ar@{|->}[r]
& [A,B]
\\
\protect{\begin{pmatrix} 0 \\ [B] \\ 0 \end{pmatrix}}\ar@{|->}[r]
& [\{0\},B]
\\
\protect{\begin{pmatrix}  0 \\ 0 \\ [A,B] \end{pmatrix}}\ar@{|->}[r]
& [\{0\}\cup A,B]
}
\qquad
\xymatrix@R=1mm{
\ol F_{m+1,n}^\omega\ar[r]^-\Psi
&\cone(\phi)
\\
[A,B]\ar@{|->}[r]
&\protect{\begin{cases}
\protect{\begin{pmatrix} [A,B] \\ 0 \\ 0 \end{pmatrix}} & \text{if $0\notin A$} \\
\protect{\begin{pmatrix} 0 \\ [B] \\ 0 \end{pmatrix}} & \text{if $A=\{0\}$} \\
\protect{\begin{pmatrix}  0 \\ 0 \\ [A-\{0\},B] \end{pmatrix}} & \text{if $A\supsetneq\{0\}$.} 
\end{cases}}
}$$
Before concluding, let's look at this for our running example.

\begin{ex}\label{ex230914a}
Continue with the situation of Example~\ref{ex230913b},
and assume that $\beta\leq\gamma\leq\delta$. 
Given the description of $\phi$ from Example~\ref{ex230913b}, we see that $\cone(\phi)$ has the following shape
where the basis vectors from the codomain of $\phi$ are listed at the top, and those from the domain are listed at the bottom.
$$\xymatrix@R=1mm{
\cone(\phi)=\!\!\!\!\!
&0\ar[r]
& S^{0+0+1}\ar[r]
&S^{1+1+3}\ar[r]
&S^{3+3+3}\ar[r]
&S^{3+3+0}\ar[r]
&0
\\
&&&[1,123]&
[1,23]&
[1,1]\\
&&&&
[1,13]&
[1,2]\\
&&&&
[1,12]&
[1,3]
\\
&&&[123]&
[23]&
[1]\\
&&&&
[13]&
[2]\\
&&&&
[12]&
[3]
\\
&&[1,123]&
[1,23]&
[1,1]\\
&&&
[1,13]&
[1,2]\\
&&&
[1,12]&
[1,3]
}
$$
Notice that the basis vectors from $\ol T$ are in rows 4--6 of basis vectors.
Compare this with our computation of $\ol F_{2,3}^\omega$ from Example~\ref{ex230913a}:
$$\xymatrix@R=1mm{
\ol F_{2,3}^\omega=
&0\ar[r]
&S^1\ar[r]^{\partial^{F_{2,3}^\omega}_4}
&S^5\ar[r]^{\partial^{F_{2,3}^\omega}_3}
&S^9\ar[r]^{\partial^{F_{2,3}^\omega}_2}
&S^6\ar[r]
&0\\
&&&
[1,123]&
[1,23]&
[1,1]\\
&&&&
[1,13]&
[1,2]\\
&&&&
[1,12]&
[1,3]\\
&&&[0,123]&
[0,23]&
[0,1]\\
&&&&
[0,13]&
[0,2]\\
&&&&
[0,12]&
[0,3]\\
&&[01,123]&[01,23]&
[01,1]&
\\
&&&[01,13]&
[01,2]&
\\
&&&[01,12]&
[01,3].&
}
$$

$$\xymatrix@R=1mm{
\ol F_{2,3}^\omega=
&0\ar[r]
&S^1\ar[r]^{\partial^{F_{2,3}^\omega}_4}
&S^5\ar[r]^{\partial^{F_{2,3}^\omega}_3}
&S^9\ar[r]^{\partial^{F_{2,3}^\omega}_2}
&S^6\ar[r]
&0\\
&&&
[1,123]&
[1,23]&
[1,1]\\
&&&&
[1,13]&
[1,2]\\
&&&&
[1,12]&
[1,3]\\
&&&[0,123]&
[0,23]&
[0,1]\\
&&&&
[0,13]&
[0,2]\\
&&&&
[0,12]&
[0,3]\\
&&[01,123]&[01,23]&
[01,1]&
\\
&&&[01,13]&
[01,2]&
\\
&&&[01,12]&
[01,3].&
}
$$
The arrangement of the basis vectors makes it straightforward to see that $\Phi$ and $\Psi$ describe inverse bijections between the bases
of $\cone(\phi)$ and $\ol F_{2,3}^\omega$, hence they describe isomorphisms between the free modules. One checks readily that they are also chain maps,
hence isomorphisms; for instance, we verify this for the basis vector $\left(\begin{smallmatrix}0\\ 0\\ [1,12]\end{smallmatrix}\right)\in\cone(\phi)$,
where $[1,12]\in (\ol F_{1,3}^\mu\otimes_{S'}S)(-X_0^\alpha)$:
$$\xymatrix@C=19mm{
\protect{\begin{pmatrix}0\\ 0\\ [1,12]\end{pmatrix}}\ar@{|->}[r]^-{\partial^{\cone(\phi)}}\ar@{|->}[d]_-{\Phi}
& \protect{\begin{pmatrix}X_0^\alpha[1,12]\\ -X_1^\gamma Y_1^{\beta-\alpha}Y_2^{\gamma-\alpha}[12]\\ Y_1^\beta[1,2]-X_1^{\gamma-\beta}Y_2^\gamma[1,1]\end{pmatrix}}\ar@{|->}[d]^-{\Phi}
\\
[01,12]\ar@{|->}[rd]_-{\partial^{\ol F_{2,3}^\omega}}
&\protect{\begin{matrix}X_0^\alpha[1,12]- X_1^\gamma Y_1^{\beta-\alpha} Y_2^{\gamma-\alpha} [0,12]\\
+Y_1^\beta[01,2]-X_1^{\gamma-\beta}Y_2^\gamma[01,1]\end{matrix}}\ar@{=}[d]\\
& \protect{\begin{matrix}\frac{X_0^\alpha X_1^\gamma Y_1^\beta Y_2^\gamma}{X_1^\gamma Y_1^\beta Y_2^\gamma}[1,12]
- \frac{X_0^\alpha X_1^\gamma Y_1^{\beta} Y_2^{\gamma}}{X_0^\alpha Y_1^{\alpha} Y_2^{\alpha}} [0,12]\\
+\frac{X_0^\alpha X_1^\gamma Y_1^{\beta} Y_2^{\gamma}}{X_0^\alpha X_1^{\gamma} Y_2^{\gamma}}[01,2]
-\frac{X_0^\alpha X_1^\gamma Y_1^{\beta} Y_2^{\gamma}}{X_0^\alpha X_1^{\beta} Y_1^{\beta}}[01,1]\end{matrix}}
}$$
\end{ex}

We now conclude our general proof of Step~\ref{step230913d}.
Arguing as in Example~\ref{ex230914a}, one checks readily that the maps $\Phi$ and $\Psi$ are inverse isomorphisms between $\cone(\phi)$ and $\ol F_{m+1,n}^\omega$.
Thus, to prove that $\ol F_{m+1,n}^\omega$ is a resolution, it remains only to show that $\cone(\phi)$ is acyclic, i.e., $\HH_i(\cone(\phi))=0$ for all $i\geq 1$.
To this end, we analyze the long exact sequence associated to $\cone(\phi)$.

Recall that, given a chain map $\psi\colon W\to U$, we have
a short exact sequence of chain complexes
$$0\to U\to\cone(\psi)\to\shift W\to 0$$
where $\shift W$ is a shifted copy of $W$. 
The associated long exact sequence in homology has the form
\begin{equation}
\cdots\to\HH_i(W)\xra{\HH_i(\psi)}\HH_i(U)\to\HH_i(\cone(\psi))\to\HH_{i-1}(W)\to\cdots.
\label{eq230915a}
\end{equation}
In our specific situation, we have chain map
$$\phi\colon (\ol F_{m,n}^\mu\otimes_{S'}S)(-X_0^\alpha)
\to(\ol F_{m,n}^\mu\otimes_{S'}S)\oplus \ol T$$
and short exact sequence
$$
0\to(\ol F_{m,n}^\mu\otimes_{S'}S)\oplus \ol T
\to\cone(\phi)\to \shift((\ol F_{m,n}^\mu\otimes_{S'}S)(-X_0^\alpha))\to 0.$$
By our induction hypothesis, the complexes $\ol F_{m,n}^\mu\otimes_{S'}S$ and $(\ol F_{m,n}^\mu\otimes_{S'}S)(-X_0^\alpha)$ are acyclic,
as is $\ol T$ since it is a Taylor resolution. 
Thus, for $i\geq 2$, the long exact sequence~\eqref{eq230915a} implies that $\HH_i(\cone(\phi))=0$,
and part of the remainder of the long exact sequence has the following form.
$$0\to \HH_1(\cone(\phi))\to\HH_0((\ol F_{m,n}^\mu\otimes_{S'}S)(-X_0^\alpha))\xra{\HH_0(\phi)}\HH_0(\ol F_{m,n}^\mu\otimes_{S'}S)\oplus \HH_0(\ol T)$$
It is straightforward to check that this has the following form
$$0\to \HH_1(\cone(\phi))\to \edgeIdeal[K_{m,n}^\mu]\xra{\left(\begin{smallmatrix}X_0^\alpha\\ \HH_0(X_0^\alpha f)\end{smallmatrix}\right)}\edgeIdeal[K_{m,n}^\mu]\oplus \edgeIdeal[K_{1,n}^\nu]$$
The top entry in the matrix here shows that the matrix describes an injective map, so $\HH_1(\cone(\phi))=0$.
This concludes the proof of Step~\ref{step230913d} and hence the proof of Theorem~\ref{thm230730a}.
\qed
\end{step}

\begin{cor}\label{cor240505a}
If condition~\ref{thm230730a1} or~\ref{thm230730a2} of Theorem~\ref{thm230730a} is satisfied, then the $(k+1)$st Betti number of 
$S/\edgeIdeal[K_{m,n}^\omega]$ is
$$\beta^S_{k+1}(S/\edgeIdeal[K_{m,n}^\omega])
=\sum_{j=1}^{k+1}\binom nj\binom m{k-j+2}.$$
\end{cor}

\begin{proof}
This follows directly from Theorem~\ref{thm230730a} and Visscher's~\cite[Corollary~6]{MR2262383}.
\end{proof}

\section{Concluding Remarks}\label{sec230915a}

We are interested to know:

\begin{question}\label{q230915a}
If conditions~\ref{thm230730a1}--\ref{thm230730a2} of Theorem~\ref{thm230730a} fail, what does the minimal free resolution of $\edgeIdeal[K_{m,n}^\omega]$ look like?
Is it cellular? Is it supported on a complex containing $V_{m,n}$?
Must we always have 
$$\beta^S_{k+1}(S/\edgeIdeal[K_{m,n}^\omega])
\geq\sum_{j=1}^{k+1}\binom nj\binom m{k-j+2}?$$
\end{question}

Data we have collected suggest that the resolution is cellular, supported on a subdivision of $V_{m,n}$,
like in the following sketch where the single 3-cell of $V_{2,3}$ is divided in two, into a solid tetrahedron and a square-based pyramid.

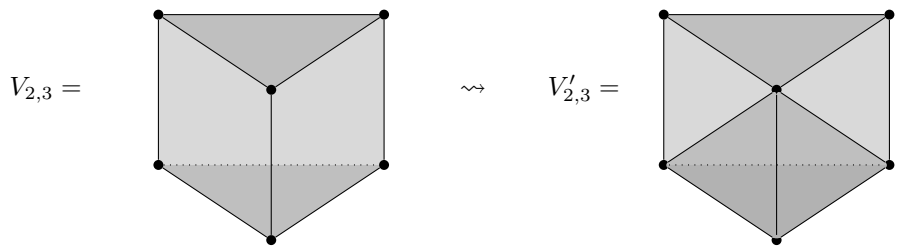
\begin{tikzpicture}
     \node (L) at (-1.5, 1) {$V_{2,3}= $ };
     
     \fill[gray!30, rounded corners=2pt] (0,0) rectangle (3,2);
     \fill[gray!50, rounded corners=2pt] (0,2) -- (1.5,1) -- (3,2) -- cycle;      
     \fill[gray!50, rounded corners=2pt] (0,0) -- (1.5,-1) -- (3,0) -- cycle;

    \node[fill, circle, scale=.4] (A) at (0, 2) {};
    \node[fill, circle, scale=.4] (B) at (3, 2) {};
    \node[fill, circle, scale=.4] (C) at (0, 0) {};
    \node[fill, circle, scale=.4] (D) at (3, 0) {};
    \node[fill, circle, scale=.4] (E) at (1.5, 1) {};
    \node[fill, circle, scale=.4] (F) at (1.5, -1) {};

    \draw (A) -- (B);
    \draw (A) -- (C);
    \draw (B) -- (D);
    \draw[dotted] (C) -- (D);
    
     \draw (A) -- (E);
    \draw (B) -- (E);
    \draw (C) -- (F);    
     \draw (E) -- (F);
    \draw (D) -- (F);
\end{tikzpicture}  \begin{tikzpicture}
     \node (L) at (-2., 1) {\qquad$\rightsquigarrow $ \qquad $V_{2,3}'= $};
     
     \fill[gray!30, rounded corners=2pt] (0,0) rectangle (3,2);
     \fill[gray!50, rounded corners=2pt] (0,2) -- (1.5,1) -- (3,2) -- cycle;

    \node[fill, circle, scale=.4] (A) at (0, 2) {};
    \node[fill, circle, scale=.4] (B) at (3, 2) {};
    \node[fill, circle, scale=.4] (C) at (0, 0) {};
    \node[fill, circle, scale=.4] (D) at (3, 0) {};
    \node[fill, circle, scale=.4] (E) at (1.5, 1) {};
    \node[fill, circle, scale=.4] (F) at (1.5, -1) {};
\fill[gray!50, rounded corners=2pt] (1.5,1) -- (1.5,-1) -- (0,0) -- cycle;
\fill[gray!50, rounded corners=2pt] (1.5,1) -- (1.5,-1) -- (3,0) -- cycle;
\fill[gray!60, rounded corners=2pt] (0,0) -- (1.5,-1) -- (3,0) -- cycle;

    \draw (A) -- (B);
    \draw (A) -- (C);
    \draw (B) -- (D);
    \draw[dotted] (C) -- (D);
    
     \draw (A) -- (E);
    \draw (B) -- (E);
    \draw (C) -- (F);    
     \draw (E) -- (F);
    \draw (D) -- (F);

\draw (E) -- (C);
\draw (E) -- (D);

\end{tikzpicture}\\
However, we are nowhere near understanding this. We are also interested to know:

\begin{question}\label{q230915b}
If conditions~\eqref{thm230730a1} and~\eqref{thm230730a2} of Theorem~\ref{thm230730a} are satisfied, does $F_{m,n}^\omega$ have the structure of a
differential graded algebra (DGA)?
\end{question}

Given the brief nature of this discussion, we direct the reader to~\cite{avramov:ifr,beck:sgidgca} for background on DGAs.

Geller~\cite[Corollary 4.6.9]{MR4336793} answers Question~\ref{q230915b} affirmatively in the square-free (i.e., the unweighted) case and, moreover, in the vertex-weighted case,
e.g., the case where $\omega$ is constant. 
It is also true if $\dim(V_{m,n})\leq 3$ because Buchsbaum and Eisenbud~\cite{buchsbaum:asffr} show that it is always true for resolutions of length at most 3. 
Beyond that, we only have guesses, even for $F_{2,4}^\omega$.

\section*{Acknowledgments}
We are grateful to the referee for helpful suggestions.

\providecommand{\bysame}{\leavevmode\hbox to3em{\hrulefill}\thinspace}
\providecommand{\MR}{\relax\ifhmode\unskip\space\fi MR }
\providecommand{\MRhref}[2]{%
  \href{http://www.ams.org/mathscinet-getitem?mr=#1}{#2}
}
\providecommand{\href}[2]{#2}

\end{document}